\documentclass[10pt]{scrartcl}
\usepackage[OT1]{fontenc}
\usepackage{ulem}
\usepackage{amsmath}
\usepackage{amssymb}
\usepackage{amsthm}
\usepackage{todonotes}
\usepackage{mathtools}
\usepackage{mathpazo}
\usepackage[mathpazo]{flexisym}
\usepackage{breqn}
\usepackage{nomencl}
\usepackage{tablefootnote}
\usepackage{ulem}
\usepackage{hyperref}
\usepackage{cleveref}
\usepackage[most]{tcolorbox}
\usepackage{tikz-cd}
\usepackage{booktabs}
\usepackage{lipsum}
\usepackage{amsfonts}
\usepackage{graphicx}
\usepackage{epstopdf}
\usepackage{algorithmic}
\ifpdf
  \DeclareGraphicsExtensions{.eps,.pdf,.png,.jpg}
\else
  \DeclareGraphicsExtensions{.eps}
\fi

\usepackage{orcidlink}

\setkomafont{title}{\normalfont\Large\bfseries}
\setkomafont{author}{\normalfont\normalsize}

\usepackage[headsepline]{scrlayer-scrpage}
\title{High order biorthogonal functions for the discrete de Rham complex on simplices\footnote{Submitted to the editors \today.}\footnote{The work of the second author is funded by the Deutsche Forschungsgemeinschaft (DFG) under Germany’s Excellence Strategy within the Cluster of Excellence PhoenixD (EXC 2122, Project ID 390833453).}}

\author{Tim Haubold\orcidlink{0000-0001-8380-0684}\thanks{Corresponding Author, CERMICS, CNRS, ENPC, Institut Polytechnique de Paris, F-77455 Marne-la-Vall\'ee cedex 2, and
Centre Inria de Paris, 48 rue Barrault, CS, 61534 F-75647 Paris, France 
  (\href{mailto:tim.haubold@enpc.fr}{tim.haubold@enpc.fr}).}
\and Sven Beuchler\orcidlink{0000-0001-9411-8701} \footnotemark[4]\thanks{Institute for Applied Mathematics and Cluster of Excellence PhoenixD, Leibniz University Hannover, Welfengarten 1, D-30167 Hannover, Germany}
\and Joachim Sch\"oberl\orcidlink{0000-0002-1250-5087}\thanks{Institute for Analysis and Scientific Computing, TU Wien, Wiedner Hauptstraße 8–10, AT-1040 Wien, Austria}}
\date{}
\newtheorem{theorem}{Theorem}[section]
\newtheorem{corollary}[theorem]{Corollary}
\newtheorem{lemma}[theorem]{Lemma}
\newtheorem{remark}[theorem]{Remark}

\usepackage{amsopn}

\usepackage[most]{tcolorbox}
\usepackage{tikz-cd}
\tikzset{%
  symbol/.style={
    draw=none,
    every to/.append style={
      edge node={node [sloped, allow upside down, auto=false]{$#1$}}
    },
  },
}
\usepackage{algorithm}
\usepackage{aligned-overset}
\usepackage{algorithmic}
\usepackage{epsfig}
\usepackage{makecell}
\pgfmathsetmacro{\a}{-2}%
\pgfmathsetmacro{\b}{2}
\pgfmathsetmacro{\h}{3}%

\usetikzlibrary{calc}%

\newcommand{\mathsym}[1]{{}}
\newcommand{\unicode}[1]{{}}

\newcommand{\lhat}{\widehat{L}}
\newcommand{\phat}{\widehat{P}}

\newcommand{\Grad}{\nabla}
\DeclareMathOperator{\Curl}{curl}
\DeclareMathOperator{\Div}{div}

\newcommand{\dx}{\, \mathrm{d}x}

\newcommand{\dz}{\, \mathrm{d}z}

\newcommand{\R}{\mathbb{R}}

\newcommand{\ned}{\mathcal{N}}
\newcommand{\RT}{\mathcal{RT}}

\newcommand{\AntGrad}{\nabla\!\!\!\!\nabla}

\newcommand{\nedelec}{N\'ed\'elec}

\newcommand{\Po}{\mathcal{P}}
\newcommand{\checka}[1]{\textcolor{blue}{#1}}  

\usepackage{setspace}
\usepackage{geometry,yfonts}
\ifpdf\hypersetup{
  pdftitle={High order biorthogonal functions for the discrete De-Rham complex},
  pdfauthor={T. Haubold and S. Beuchler and J. Sch\"oberl}
}
\fi

\begin{document}

\maketitle
\begin{abstract}
It is well known that the choice of basis functions in $hp$-FEM heavily influences the stability and the computational cost in order to obtain an approximate solution.
For simplicial elements in two and three space dimensions, tensor-product-like basis functions built from Jacobi polynomials with different weights yield optimal properties with respect to condition number and sparsity. In this paper we construct such a high order basis for the Nédélec spaces and for the Raviart–Thomas and Brezzi–Douglas–Marini spaces, modifying existing definitions of Zaglmayr so that the functions belong to the Nédélec spaces of first kind and the Raviart-Thomas spaces. The bases are designed so that the Nédélec space of second kind and the BDM space are extensions of the respective other spaces.
In the second part of the paper we introduce biorthogonal basis functions for $H(\Div,\Omega)$ continuing previous research for $H^1(\Omega)$ and $H(\Curl,\Omega)$. These functions can be expressed in closed form as sums of tensor products of Jacobi polynomials, which allows for fast computation of the $L^2$-projection.

\end{abstract}

\section{Introduction}

It is well known that $hp$ finite element methods (FEM) often exhibit exponential convergence rates depending on the data, see e.g. \cite{szabo1991,schwab1998,melenk2002} and also \cite{karniadakis2013} for the related spectral element method
under the condition that the exact solution of the
underlying partial differential equation is (locally) sufficiently smooth.
A very important algorithmic ingredient of a $hp$-FEM method is the choice of basis functions
which directly influences the condition number of the involved matrices, see \cite{BabGriPit89,maitre1996}.

An alternative to nodal basis functions are modal basis functions, e.g. based on Bernstein polynomials,
see \cite{ainsworth2011,ainsworth2018}, 
or classical orthogonal polynomials, which go back to Szab\'{o} and Babu\v{s}ka, see e.g. \cite{szabo1991}. 
The latter choice yields overall better condition numbers. Furthermore, in the case of a polygonal domain and piecewise constant material functions, they yield  sparse element matrices and can be assembled in optimal complexity as well, see \cite{beuchler2023}, see also \cite{EibMel06}.
Basis functions for simplices are introduced by \cite{Dubiner,karniadakis2013}, see also the new approach
\cite{brubeck2026fastsolvershighorderfem} based on local eigenvalue solves.
Due to different variational formulations, one naturally derives the different function spaces $H^1,H(\Curl)$ and $H(\Div).$ These spaces are connected by an exact sequence, the \textit{de Rham complex}.
It describes the relation of the Sobolev spaces $H^1(\Omega), H(\Curl,\Omega)$ and $H(\Div,\Omega)$. In the following we assume that $\Omega \subset \R^3$ is a simply connected Lipschitz domain and that $\partial \Omega$ is connected. We introduce  
\begin{align*}
H^1(\Omega) &\coloneqq \lbrace u \in L^2(\Omega)\vert \Grad u \in L^2(\Omega) \rbrace,\\
H(\Curl,\Omega) &\coloneqq \lbrace u \in L^2(\Omega)\vert \Curl u \in L^2(\Omega) \rbrace,\\
H(\Div, \Omega) &\coloneqq \lbrace u \in L^2(\Omega)\vert \Div u \in L^2(\Omega) \rbrace.
\end{align*}
If $u \in H^1(\Omega)$, then $\Grad u \in H(\Curl, \Omega)$, and if $u \in H(\Curl, \Omega)$, then $\Curl u \in H(\Div, \Omega)$. Thus, the image of the respective differential operator is in the kernel of the next operator in this sequence. This cochain complex is called the de Rham complex.\\ 
These results in $3D$ are summarized by the following diagram
\begin{equation*}
   \R \quad\overset{i}{\longrightarrow}H^1(\Omega)\quad\overset{\Grad}{\longrightarrow} \quad H(\Curl, \Omega) \quad\overset{\Curl}{\longrightarrow} \quad H(\Div, \Omega)\quad\overset{\Div}{\longrightarrow}\quad L^2(\Omega) \quad \overset{\circ}{\longrightarrow} \lbrace 0 \rbrace,
\end{equation*}
where $i$ maps a real number to a constant function and $\circ$ is the zero map. This cochain is exact, i.e. $\operatorname{im}d_k = \operatorname{ker} d_{k+1}$, for all operators $d_k \in \lbrace i, \Grad, \Curl,\Div,\circ\rbrace$, see \cite[Prop. 16.14]{ern2021}.\\
The construction of an exact polynomial sequence on a simplex $T$ has been extensively discussed in the literature, and it is well known that two different exact sequences exist, see e.g. \cite{demkowicz2006,demkowicz2008,monk2003,ern2021}. One cochain contains the full polynomial spaces, and the second one uses special polynomial spaces, which are introduced in the following.\\
For $k\in\mathbb{N}$, let $\Po_k$ be the space of polynomials of maximal degree $k$. The vector valued polynomials
are denoted as ${[\Po_k]}^d$ for a space dimension $d=2,3$.
Following \cite{monk2003}, let
\begin{equation*}
    {[\tilde{\Po}_k]}^d = \lbrace\text{homogeneous polynomials of total degree exactly }k \text{ in } \R^d\rbrace
\end{equation*}
and 
$R_k = {[\Po_{k-1}]}^d \oplus S_k$, where $S_k = \lbrace p \in {[\tilde{\Po}_{k}]}^d \vert \vec{r} \cdot p = 0 \rbrace$, where $\vec{r} = {(x,y,(z))}^\top$. The space $R_k$ is called \nedelec~space of first kind.
In addition, we introduce the Raviart-Thomas space as
\[
\RT_{k}=\vec{r} \tilde{\Po}_{k-1} \oplus {[\Po_{k-1}]}^d
\]
or equivalently
\[
\RT_{k}=\vec{r} \Po_{k-1} + {[\Po_{k-1}]}^d.
\]
The cochain for the \nedelec~and Raviart-Thomas space and for the full polynomial spaces, i.e. the \nedelec~space of type two and BDM space are depicted in \Cref{tab:DeRham}. This cochain is exact, see e.g. \cite{monk2003,ern2021}.
\begin{table}[h]
    \centering
    \begin{tabular}{c c c c c c c c}
      \toprule
        & $H^1(T)$& &$H(\Curl,T)$& &$H(\Div,T)$& &$L^2(T)$\\
      \midrule
      First sequence &$\Po_{k}$ & $\xrightarrow{\Grad}$ & $R_k$ & $\xrightarrow{\Curl}$ & $\RT_{k}$ & $\xrightarrow{\Div}$ & $\Po_{k-1}$ \\
      Second sequence &$\Po_{k+2}$ & $\xrightarrow{\Grad}$ & ${[\Po_{k+1}]}^3$ & $\xrightarrow{\Curl}$ & ${[\Po_{k}]}^3$ & $\xrightarrow{\Div}$ & $\Po_{k-1}$ \\
      \bottomrule
    \end{tabular}
    \caption{Discrete de Rham Complex on a tetrahedron $T$}\label{tab:DeRham}
    \end{table}
Furthermore, the continuous and discrete de Rham complex fulfill a commuting property. This means there exist interpolation operators $I^g_T, I^c_T, I^{d}_T$ and $I^b_T$, such that 
\[
    \Grad I^g_T (u) = I^c_T \Grad (u), \quad \Curl I^c_T (u) = I^d_T \Curl (u),\quad \Div I^d_T (u) = I^b_T \Div (u), 
\]
where $I^b_T$ denotes the $L^2$ orthogonal projector onto $\mathcal{P}_k(T)$, see \cite{ern2021}.
An example of such interpolation operators is the projection-based interpolation operators introduced in \cite{demkowicz2005}, but see also \cite{bespalov2009,melenk2019}. For interpolation operators that are optimal in mesh-size $h$ and polynomial degree $k$, see \cite{ern2022,chaumont-frelet2024}. In most of these references an orthogonal projection is used to define the interpolation operator. \\
Different constructions of bases for the polynomial spaces of the discrete de Rham complex have been given. Basis functions based on orthogonal polynomials for the first sequence can be found, e.g. in~\cite{fuentes2015}, and for the second 
sequence in~\cite{ainsworth2001,zaglmayr2006, beuchler2013a, beuchler2012a, beuchler2012}. 

In this paper, we will modify the basis functions based on the work by~\cite{zaglmayr2006, beuchler2012}
such that the basis of the second sequence contains the first sequence. 

From an implementation point of view, this is a very desirable property, since we now only need to enable or disable certain degrees of freedom in our code to switch between spaces in the first or second sequence.
The basis functions \cite{zaglmayr2006} are divided into two groups. 
For $H(\Curl,\Omega)$ in $3D$, the first one is built by the application of the gradient to the basis
functions $u_{ijk}$ of the previous space, in this case $H^1$, where the indices $i,j,k$ stand for the polynomial degrees of the involved basis functions in the different directions.
The second one is built by a different linear combination of  the contributions 
$\frac{\partial}{\partial_{x_s}} u_{ijk}$, $s=1,2,3$.
In \cite{zaglmayr2006}, the factors in the linear combination are chosen independently of $i,j$ and $k$.
In this paper, we will show that a different choice of the linear combinations depending on $i,j,$ and $k$
will ensure that the basis of the second sequence contains the first sequence.
The same can also be done for the other spaces of the de Rham complex in two and three space dimensions.
The proof uses just a simple property of polynomials of degree $i$, namely \eqref{eq:DiffFormel}.
Therefore, this modification is applicable not only to basis functions as in \cite{beuchler2013} with certain 
Jacobi polynomials  for the involved one-dimensional auxiliary functions,
but also to very general polynomials. 
Therefore, the notation in \Cref{sec:newned} is kept quite general.

The second novelty of this paper is an extension of the construction of biorthogonal function for the space
$H(\Curl,\Omega)$ in \cite{haubold2024} to $H(\Div,\Omega)$.
For a short motivation we refer to the introduction of \cite{haubold2024} or to \cite{BanzSchroed15,BAMMER2025} 
with application to contact problems and elastoplasticity.
In this paper, the primal basis consists of the functions \cite{beuchler2012} which are orthogonal in the $H(\Div,\Omega)$
inner product. The construction principle is similar to \cite{haubold2024}.
However, the proof of all orthogonality relations is quite technical and involves also some
more special relations about Jacobi polynomials, which we will summarize in the following section.
We also refer to \cite{LamWohl07} for biorthogonal functions for Lagrange polynomials.

The paper is organized as follows. 
Section \ref{ch:Jac+Basis} starts with collection of properties of Jacobi and integrated Jacobi polynomials.
Some integral relations which are needed in the following are also proved.
In the last subsection, we summarize the construction of the element based basis functions 
for all spaces of the de Rham complex in two and three space dimensions.
Then in \Cref{sec:newned} and \Cref{ch:Div+Basis} we introduce new basis functions for the \nedelec-space of second kind and the BDM space, where the basis functions have a clear separation in functions of the \nedelec-space of first kind and gradients of $H^1$-functions, and $\RT$ functions and solenoidal fields, respectively.
In \Cref{sec:Biorthogonal}, we introduce functions, which are biorthogonal to the BDM-space defined at the end of \Cref{ch:Div+Basis}. 
\section{Preliminaries}\label{ch:Jac+Basis}
\subsection{Jacobi polynomials}
We start with the definition of the required orthogonal polynomials.
For $n\in\mathbb{N}$, $\alpha,\beta>-1$, let 
\begin{equation*}
    P_{n}^{(\alpha,\beta)}(x)=\frac{1}{2^n n!{(1-x)}^\alpha{(1+x)}^\beta} 
    \frac{\mathrm{d}^n}{\mathrm{d}x^n} {(1-x)}^\alpha{(1+x)}^\beta {(x^2-1)}^n
\end{equation*}
be the $n$-Jacobi polynomial with respect to the weight $\omega(x)={(1-x)}^\alpha{(1+x)}^\beta$.
Moreover, the integrated Jacobi polynomials are given by
\begin{equation*}
    \hat{P}_{n}^{\alpha}(x)=\int_{-1}^x  P_{n-1}^{(\alpha,0)} (t) \;\mathrm{d}t, \quad 
    \hat{P}_0^{\alpha}(x) = 1
\end{equation*}
for $n\geq 1$ and $\beta=0$.
In the special case $\alpha=\beta=0$, one obtains 
\begin{equation*}
    L_n(x)=P^{(0,0)}_n(x) \quad\textrm{and}\quad \hat{L}_n(x)=\hat{P}_n^0(x)
\end{equation*}
the Legendre and integrated Legendre polynomials, respectively.
The Jacobi polynomials form an orthogonal system in the weighted $L_{2,\omega}$ scalar product
\begin{equation}
\label{Orthogonality2}
I_{n,m}^{(\alpha,\beta)}=\int_{-1}^1 \omega(x) P_{n}^{(\alpha,\beta)}(x)P_{m}^{(\alpha,\beta)}(x) \;\mathrm{d}x
= \delta_{nm} \frac{{2}^{\alpha+\beta+1}}{(2n+\alpha+\beta+1)} 
\frac{\Gamma(n+\alpha+1)\Gamma(n+\beta+1)}{n! \Gamma(n+\alpha+\beta+1)},
\end{equation}
where $\Gamma(\cdot)$ denotes the Gamma function,
see e.g.~\cite{andrews1999}.
In this publication, the cases $\beta\in\{0,1\}$ are of special interest.
Then, the relation~\eqref{Orthogonality2} simplifies to
\begin{equation}
    \label{Orthogonality}
    I_{n,m}^{(\alpha,0)}=\delta_{nm} \frac{2^{\alpha+1}}{(2n+\alpha+1)} \quad\textrm{and}\quad
    I_{n,m}^{(\alpha,1)}=\delta_{nm} \frac{2^{\alpha+2}}{(2n+\alpha+2)} \frac{(n+1)}{(n+\alpha+1)},
\end{equation}
respectively.
Note that the integrated Legendre and Jacobi polynomials can be written as Jacobi polynomials by
using the relations, 
\begin{equation}\label{dual:IntLeg}
    \lhat_i(x) = \frac{(x^2-1)}{2(i-1)} P_{i-2}^{(1,1)}(x)
    \quad\textrm{and}\quad  \phat^{\alpha}_i(x) = \frac{(1+x)}{i} P^{(\alpha - 1,1)}_{i-1}(x), \alpha \geq 1
\end{equation}
see e.g.~\cite{szego1967}.
The identity
\begin{equation}\label{eq:JacDiff}
P_k^{(\alpha,1)}(x) = \frac{2}{(k+\alpha+1)} \frac{\mathrm{d}}{\mathrm{d}x} \, P_{k+1}^{(\alpha-1,0)}(x)
\end{equation}
is also very important, \cite[Eq. (4.21.7)]{szego1967}.
In the following, we will need the relation
\begin{equation}\label{eq:NedJac}
        \frac12(2+\alpha+\beta+2j)(x-1)P_j^{(\alpha+1,\beta)}(x)=(j+1)P_{j+1}^{(\alpha,\beta)}(x)
        -(j+1+\alpha)P_j^{(\alpha,\beta)}(x),\\
\end{equation}
see e.g.~\cite[Ch. 16, \S 138.16]{rainville1960}.
From~\cite[Eq.(19)]{beuchler2007} the relation
\begin{equation}\label{eq:Rec_BP}
    \eta P_{j-1}^{(\alpha,0)}(\eta) - j \phat_j^\alpha (\eta) = \frac{1}{2j+\alpha-2} {\left(-\alpha P_{j-1}^{(\alpha,0)}(\eta) + (2j-2) P_{j-2}^{(\alpha,0)}(\eta)\right)}
\end{equation}
is known. In the special case $\alpha=0$ one obtains
\begin{equation}\label{eq:Rec_BS}
    i \, \lhat_i(x) = x L_{i-1}(x) - L_{i-2}(x),
\end{equation}
see~\cite{beuchler2006}.
Other relations of interest are 
\begin{eqnarray}
        -(1-y)\left(y P^{(3,0)}_{j-1} (y)- j \phat^3_j(y)\right) + 2 \phat_j^3 (y)&=& \frac{4(j+1)}{2j+1} P_{j-1}^{(1,0)}(y) + \frac{2}{2j+1} P_j^{(1,0)}(y) \quad\textrm{and}\quad \label{eq:BeuPill17a}\\
        \phat^3_j(y)-\frac{1-y}{2} P^{(3,0)}_{j-1} (y)&=&  P^{(1,0)}_{j} (y), \label{eq:BeuPill17}
        \end{eqnarray}
see  \eqref{eq:Rec_BP} together with \cite[Eq. (18)]{beuchler2007}, and
\cite[Eq. (17)]{beuchler2007} for $\alpha=3$, respectively.

Finally, we mention Lemma 4.7 from~\cite{haubold2024}:
\begin{lemma}{}{}\label{lemma:IntJacLeg}
For $1 \leq i,k$, the relation
\begin{equation*}
\int_{-1}^{1} L_i(x) P_k^{(1,1)}(x)\dx = \begin{cases} \frac{4}{2+k} & \text{ if } k\geq i \text{ and } (k-i) \mod 2 = 0 \\0 & \text{else} \end{cases}
\end{equation*}
holds.
\end{lemma}
This result has several consequences, which we will later use in this paper.
\begin{corollary}{}{}\label{lemma:ProdIntJacLeg}
    The relations
    \begin{equation*}
    \begin{aligned}
    \int_{-1}^1 -k P_{k-2}^{(1,1)}(x) L_{i-2}(x)+(k+1) P_{k-1}^{(1,1)}(x)L_{i-1}(x) \;\mathrm{d}x&=0, \quad \forall i,k\geq 2\\
    \int_{-1}^1 -k P_{k-2}^{(1,1)}(x) L_{i}(x)+(k+1) P_{k-1}^{(1,1)}(x)L_{i-1}(x) \;\mathrm{d}x&=0, \quad \forall k\geq 2,i \geq 1, i\neq k
    \end{aligned}
    \end{equation*}
    hold.
\end{corollary}
\begin{proof}
Let us start with the first relation.
    We distinguish between three cases. 
    \begin{enumerate}
        \item     If $i> k$, the orthogonality of the Legendre polynomials proves the result directly.
        \item Due to symmetry with respect to $x=0$ the integrals of both summands are zero for $(k-i) \mod 2 =1$.
        \item In the remaining case, Lemma \ref{lemma:IntJacLeg} gives
        \[
        \int_{-1}^1 -k P_{k-2}^{(1,1)}(x) L_{i-2}(x)+(k+1) P_{k-1}^{(1,1)}(x)L_{i-1}(x) \;\mathrm{d}x
        =-k \frac{4}{k}+(k+1)\frac{4}{k+1}=0.
        \]
        \end{enumerate}
This proves the assertion.
The second relation follows by the same arguments.
\end{proof}
\begin{remark}
    This is a simplification of a recursion formula in \cite[Corollary 7]{beuchler2023} by using the parity of the symmetric Jacobi polynomials. 
\end{remark}
\begin{lemma}
Let $\alpha>0$. Then,
    \[
        \int_{-1}^{1} (1-x)^\alpha P_j^{(\alpha,0)} (x)P_k^{(\alpha,1)}(x)\dx = \begin{cases}
            0 & \text{if}\quad j > k\\
            \frac{2^{\alpha+1}(-1)^{j+k}}{k+\alpha+1} & \text{if}\quad j \leq k
        \end{cases}  
    \]
\end{lemma}
\begin{proof}
The case $j>k$ follows from the orthogonality \eqref{Orthogonality} of the Jacobi polynomials.
Therefore, it suffices to consider the case $k \geq j$ only.
Then integration by parts together with \eqref{eq:JacDiff} is used to obtain
    \begin{align*}
                \int_{-1}^{1} (1-x)^\alpha P_j^{(\alpha,0)}(x) P_k^{(\alpha,1)}(x) \dx =& \frac{2}{(k+\alpha+1)}(1-x)^\alpha P_{j}^{(\alpha,0)}(x)P_{k+1}^{(\alpha-1,0)}(x){\bigg|}_{-1}^{1}\\ &- \frac{2}{(k+\alpha+1)}\int_{-1}^{1} \frac{\mathrm{d}}{\dx}\left((1-x)^\alpha P^{(\alpha,0)}_{j}(x)\right) P_{k+1}^{(\alpha-1,0)}(x).
    \end{align*}
    Note that 
    $\frac{\mathrm{d}}{\dx}\left((1-x)^\alpha P^{(\alpha,0)}_{j}(x)\right)=(1-x)^{\alpha-1} q_j(x)$
    with some polynomial $q_j\in\Po_j$. By the orthogonality \eqref{Orthogonality} 
    of the Jacobi polynomials for $k\geq j$, we get
    \[
     \int_{-1}^{1} (1-x)^\alpha P_j^{(\alpha,0)}(x) P_k^{(\alpha,1)}(x) \dx = \frac{2}{(k+\alpha+1)}(1-x)^\alpha P_{j}^{(\alpha,0)}(x)P_{k+1}^{(\alpha-1,0)}(x){\bigg|}_{-1}^{1}=\frac{2^{\alpha+1} (-1)^{j+k}}{k+\alpha+1}
    \]
    by observing that $P^{(\alpha,0)}_{j}(-1) =(-1)^{j}$.
    \end{proof}
    Using the result for $\alpha=1$ and $\alpha=2$ gives
    \begin{corollary} \label{folg}
    There holds the identity
    \[
    \int_{-1}^1 (1-x) P_{j}^{(1,0)}(x) P_{k}^{(1,1)}(x) \dx-\frac12\int_{-1}^1 (1-x)^2 P_{j-1}^{(2,0)}(x) P_{k-1}^{(2,1)}(x)\dx=0 \quad \textrm{for all } j,k\geq 1.
    \]
    \end{corollary}

\subsection{Basis functions}
Using some polynomials, the interior basis functions can be introduced.

We start with the triangular case with the reference triangle with the vertices $(-1,-1)$, $(1,-1)$ and $(0,1)$.
Let $q_i^x\in\Po_{i}\setminus\Po_{i-1}$, $i\geq 2$ and $q_j^{y,i}\in\Po_j\setminus\Po_{j-1}$, $j\geq 1$ be two families of polynomials of degrees $i$ and $j$, respectively, satisfying $q_i^x(\pm 1)=0$ and $q_j^{y,i}(-1)=0$.
The notation $q_j^{y,i}$ induces that the family of polynomials of degree $j$ may depend on $i$.

\paragraph{Basis function of $H(\Curl)$:}Following~\cite{zaglmayr2006}, the element based basis functions of $H(\Curl)$ 
on the reference triangle $\triangle$ are developed from the $H^1$ basis functions 
$u_{ij}^\triangle = f_i(x,y) g_{i,j}(x,y)$ where $f_i(x,y) = q_i^x\left(\frac{2x}{1-y}\right) {\left(\frac{1-y}{2}\right)}^i$ and $g_{i,j}(y) = q^y_{i,j}(y)$ by taking partial derivatives. More precisely, 
\begin{equation}\label{Dual:BasisTrig}
\begin{aligned}
    v^{\triangle, I}_{ij}(x,y) &= \Grad(u^{\triangle}_{ij}(x,y))= \Grad(f_i(x,y)) g_{i,j}(y) + f_i(x,y) \Grad(g_{i,j}(y)),\\
    v^{\triangle, II}_{ij}(x,y) &= \AntGrad_2 (u_{ij}^\triangle(x,y)) := \Grad(f_i(x,y)) g_{i,j}(y) - f_i(x,y) \Grad(g_{i,j}(y)),
     \textrm{ and}\\
    v^{\triangle,III}_{1j}(x,y) &= \Grad(f_1(x,y)) \phat^{3}_j(y).
\end{aligned}
\end{equation}
In~\cite{beuchler2012}, the choices $q_i^x=\lhat_i$ and $q_j^{y,i}=\phat_j^{\alpha_i}$ with the weight $\alpha_i=2i$ are used. 
This case leads to an optimal sparsity pattern of the mass and stiffness matrix
with respect to the number of nonzero entries.
Moreover, it is possible to introduce $L^2$ dual polynomial functions in closed formulas, \cite{haubold2024}.
The original definition of~\cite{zaglmayr2006} allows more general weights for the Jacobi polynomials in $y$-direction. With regards to sparsity the optimal choice is $\alpha_i= 2i-1$, but in this case the biorthogonal functions become rational.\\ 
In the following \Cref{sec:newned}, we will redefine the basis functions of type $II$. In doing so, we allow the possible disabling of the highest polynomial degree of the basis functions of type $I$, 
which yields the~\nedelec~basis of first kind.

For the tetrahedral case with the reference tetrahedron $\blacktriangle$ with vertices $(-1,-1,-1)$, $(1,-1,-1)$, $(0,1,-1)$ and $(0,0,1)$, the construction principle is similar.
In addition to $q_i^x$ and $q_j^{y,i}$, let $q_k^{z,i,j}\in\Po_k\setminus\Po_{k-1}$, $k\geq 1$ be a family
of polynomials satisfying $q_k^{z,i,j}(-1)=0$. Then, the element based basis functions for $H^1$ on the reference element
$\blacktriangle$ are given by $u_{ijk}^\blacktriangle=f_i g_{i,j} h_{i,j,k}$, where
 \begin{equation}\label{eq:defaux3d}
    f_i(x,y,z) = q_i^x{\left(\frac{4x}{1-2y-z}\right)} {\left(\frac{1-2y-z}{4}\right)}^i, \quad
    g_{i,j}(y,z) = q_j^{y,i}{\left(\frac{2y}{1-z}\right)} {\left(\frac{1-z}{2}\right)}^j, \textrm{ and }
    h_{i,j,k} (z)= q_k^{z,i,j}(z).
    \end{equation}
Similar to the two-dimensional case, 
the element based basis functions of $H(\Curl)$ on the tetrahedron $\blacktriangle$ are given by
\begin{equation}\label{Dual:TetBasis}
    \begin{aligned}
    v_{ijk}^{\blacktriangle,I} &= \Grad(f_i g_{i,j} h_{i,j,k}) = \Grad (f_i) g_{i,j} h_{i,j,k} + f_i \Grad(g_{i,j}) h_{i,j,k} + f_i g_{i,j} \Grad(h_{i,j,k}),\\
    v_{ijk}^{\blacktriangle,II} &= \Grad (f_i) g_{i,j} h_{i,j,k} - f_i \Grad(g_{i,j}) h_{i,j,k} + f_i g_{i,j} \Grad(h_{i,j,k}),\\
    v_{ijk}^{\blacktriangle,III} &= \Grad (f_i) g_{i,j} h_{i,j,k} + f_i \Grad(g_{i,j}) h_{i,j,k} - f_i g_{i,j} \Grad(h_{i,j,k}),\\
    v_{1jk}^{\blacktriangle,IV} &= v^{\ned_0}_{[1,2]} g_{1,j} h_{1,j,k},
    \end{aligned}
\end{equation}
for $2 \leq i, 1\leq j,k$, and $i+j+k\leq p$.
Here $v_{[1,2]}^{\ned_0}$ is the lowest order \nedelec~function of the first kind, based on the edge from vertex $1$ to $2$.
In~\cite{beuchler2012}, the choices $q_i^x=\lhat_i$, $q_j^{y,i}=\phat_j^{\alpha_i}$ and $q_{k}^{z,i,j}=\phat_{k}^{\beta_{i,j}}$
with the weights $\alpha_i=2i$ and $\beta_{i,j}=2i+2j$ for the Jacobi polynomials
are mentioned for sparsity optimality and condition number of the involved matrices. 
However also other weights in the Jacobi polynomials are also possible in the construction in~\cite{zaglmayr2006}.

\paragraph{Basis functions of $H(\operatorname{div})$}With~\eqref{Dual:TetBasis}, the divergence free element based basis functions on the tetrahedron $\blacktriangle$
are obtained by taking the curl of the basis functions of types $II$, $III$ and $IV$:
\begin{equation}\label{Div:TetBasis}
    \begin{aligned}
    w_{ijk}^{\blacktriangle,II}&:= \Curl v_{ijk}^{II} &=2 h_{i,j,k} \nabla g_{i,j} \times \nabla f_i- 2 f_i \nabla h_{i,j,k} \times \nabla g_{i,j}\\
    w_{ijk}^{\blacktriangle,III}&:=\Curl v_{ijk}^{III} &= 2 g_{i,j} \nabla h_{i,j,k} \times \nabla f_i+2 f_i \nabla h_{i,j,k} \times \nabla g_{i,j}\\
    w_{1jk}^{\blacktriangle,IV}&:= \Curl v_{1jk}^{IV} &= \Curl {\left(v^{\ned_0}_{[1,2]} g_{1,j} h_{1,j,k}\right)},
    \end{aligned}
\end{equation}
\sloppy
In~\cite{zaglmayr2006}, they are completed by the functions
\begin{equation}\label{Div:TetBasis2}
    \begin{aligned}
     w_{ijk}^{\blacktriangle,I}&:=2 f_i \nabla h_{i,j,k} \times \nabla g_{i,j},   &\quad &
    w_{1jk}^{\blacktriangle,I}&:= g_{1,j} \left(v^{\ned_0}_{[1,2]} \times \nabla  h_{1,j,k}\right)
    &\quad\textrm{and}\quad
        w_{10k}^{\blacktriangle,I}&:= 4 w_1^{\blacktriangle,F_1}(x,y,z) \widehat{P}^3_k(z),
    \end{aligned}
\end{equation}
where $w_1$ denotes the Raviart-Thomas functions associated to the bottom face $F_1$ of $\blacktriangle$. 
Removing the highest polynomial degree of the basis functions of type $II,III$ and $IV$ reduces the BDM-space to the RT-space.


\section{New \texorpdfstring{$H(\Curl)$}{} basis functions on a triangle and a tetrahedron}\label{sec:newned}
By deriving the \nedelec-functions of the first kind, we will see that the basis functions $v_{ij}^{\triangle,II}(x,y)$ on the triangle are not \nedelec-conforming. In this section, we will modify them in order to obtain a \nedelec-conforming one.
This principle uses the following observation for polynomials of exactly degree $i$, namely
\begin{equation}\label{eq:DiffFormel}
    \chi q'_i(\chi)-i q_i(\chi)=\tilde{Q}_{i-1}(\chi), \quad \forall q_i\in\Po_{i},
\end{equation}
where $\tilde{Q}_{i-1} $ is some polynomial of maximal degree $i-1$.
In the case of integrated Jacobi or Legendre polynomials, the remainder $\tilde{Q}_{i-1}$ can be expressed in
terms of Jacobi polynomials, see~\eqref{eq:Rec_BP},~\eqref{eq:Rec_BS}.
In the case of the integrated Legendre polynomials it is a polynomial of degree $i-2$.

Our constructions for the $H(\Curl)$ conforming functions follow \nedelec's construction. 
By using a modified gradient  operator in \eqref{Dual:TetBasis}, we define non-curl-free
functions and enrich this space by the gradients. By using this construction on all polynomial degrees, our basis stays hierarchical. But the interior functions, defined by this modified gradient operator, do not split into a homogeneous and a non-homogeneous part, as needed by the \nedelec~space of first kind. 
\subsection{Triangle}
For the triangle we redefine our type $v_{ij}^{\triangle,II}$ functions in~\eqref{Dual:BasisTrig} 
with general polynomials $q_i^x$ and $q_j^{y,i}$ in the auxiliary functions $f_i$ and $g_{ij}$ by
\begin{equation}\label{eq:dualnedtrig}
v_{ij}^{\triangle,\ned}(x,y) = j(\Grad f_{i}(x,y))g_{i,j}(x,y) - i f_i(x,y) \Grad g_{i,j}(x,y).\end{equation}
Those are in the first \nedelec~space $R_k$, as can be seen in the following theorem:
\begin{theorem}
Let us assume that $q_i^x$ and $q_j^{y,i}$ satisfy \eqref{eq:DiffFormel} and let 
$v_{ij}^{\triangle,\ned}(x,y)$ be defined by \eqref{eq:dualnedtrig}.
If $i+j \leq p$, with $i\geq 2, j\geq1$,
then     $v_{ij}^{\triangle,\ned}(x,y) \in R_{p-1}$.
\end{theorem}
\begin{proof}
We compute the gradients of $f_i$ and $g_{ij}$.
This gives 
\[
j\Grad f_i=\begin{pmatrix} j (q_i^x)'{\left(\frac{2x}{1-y}\right)} {\left(\frac{1-y}{2}\right)}^{i-1} \\
-\frac{j}{2}  (\tilde{Q}_{i-1}^x){\left(\frac{2x}{1-y}\right)} {\left(\frac{1-y}{2}\right)}^{i-1}
\end{pmatrix}
\quad
\textrm{and}
\quad
i \Grad g_{ij}=\begin{pmatrix}
    0 \\ i (q_j^{y,i})'(y)
\end{pmatrix}
\]
by using the chain rule and~\eqref{eq:DiffFormel} for $q_i^x$.
Hence, 
\begin{equation}\label{eq:2dT1}
    \begin{aligned}
j g_{i,j}\Grad f_i &=\begin{pmatrix}  (q_i^x)'{\left(\frac{2x}{1-y}\right)} {\left(\frac{1-y}{2}\right)}^{i-1} \\
\frac{1}{2}  (\tilde{Q}_{i-1}^x){\left(\frac{2x}{1-y}\right)} {\left(\frac{1-y}{2}\right)}^{i-1}
\end{pmatrix} j q_j^{y,i}(y) \overset{\eqref{eq:DiffFormel}}{=} 
\begin{pmatrix}  (q_i^x)'{\left(\frac{2x}{1-y}\right)} {\left(\frac{1-y}{2}\right)}^{i-1} \\
\frac{1}{2}  (\tilde{Q}_{i-1}^x){\left(\frac{2x}{1-y}\right)} {{\left(\frac{1-y}{2}\right)}}^{i-1}
\end{pmatrix} (y (q_j^{y,i})'(y)-\tilde{Q}_{j-1}^{y,i}(y))\\
&= \begin{pmatrix}  (q_i^x)'{\left(\frac{2x}{1-y}\right)} \\
\frac{1}{2}  (\tilde{Q}_{i-1}^x){\left(\frac{2x}{1-y}\right)}
\end{pmatrix} {\left(\frac{1-y}{2}\right)}^{i-1} y (q_j^{y,i})'(y)  + r_{i+j-2}
\end{aligned}
\end{equation}
with $r_{i+j-2}\in{[\Po_{i+j-2}]}^2$.
Using~\eqref{eq:DiffFormel} again
\begin{equation*}
    \begin{aligned}
        i f_i&=i q_{i}^x{\left(\frac{2x}{1-y} \right)} {\left( \frac{1-y}{2}\right)}^i
        \overset{\eqref{eq:DiffFormel}}{=} {\left( \frac{1-y}{2}\right)}^i
        {\left(\frac{2x}{1-y}  (q_{i}^x)'{\left(\frac{2x}{1-y} \right)}-\tilde{Q}_{i-1}^x {\left(\frac{2x}{1-y} \right)}  \right)}\\
        &= {\left( \frac{1-y}{2}\right)}^{i-1} 
        {\left(  x (q_{i}^x)'{\left(\frac{2x}{1-y} \right)}-\frac{1-y}{2}\tilde{Q}_{i-1}^x {\left(\frac{2x}{1-y}   \right)}\right)}.
    \end{aligned}
\end{equation*}
This gives
\begin{equation}\label{eq:2dT2}
    \begin{aligned}
        i f_i \Grad g_{ij}=  {\left( \frac{1-y}{2}\right)}^{i-1} 
        {\left(  x (q_{i}^x)'{\left(\frac{2x}{1-y} \right)}-\frac{1-y}{2}\tilde{Q}_{i-1}^x {\left(\frac{2x}{1-y}   \right)}\right)} (q_j^{y,i})'(y) \begin{pmatrix}0 \\ 1 \end{pmatrix}
    \end{aligned}
\end{equation}
Using~\eqref{eq:2dT1} and~\eqref{eq:2dT2}
\begin{equation*}
    \begin{aligned}
        i f_i \Grad g_{ij}-j g_{i,j}\Grad f_i&=\underbrace{(q_{i}^x)'{\left(\frac{2x}{1-y} \right)}
        {\left(\frac{1-y}{2}\right)}^{i-1}(q_j^{y,i})'(y)}_{\in\Po_{i+j-2}} \begin{pmatrix} y \\ -x\end{pmatrix} \\
        &\quad +\underbrace{r_{i+j-2}-
        \frac12 \tilde{Q}_{i-1}^x{\left( \frac{2x}{1-y}\right)} {\left(\frac{1-y}{2}\right)}^{i-1} (q_j^{y,i})'(y) 
        \begin{pmatrix}0 \\ 1 \end{pmatrix}}_{\in{[\Po_{i+j-2}]}^2}.
       \end{aligned}
\end{equation*}
This proves the theorem.
\end{proof}
\subsection{Tetrahedron}
The same property holds for the three-dimensional case on the tetrahedron by using the basis functions~\eqref{Dual:TetBasis}.
Instead of proving everything in one go, as for the $2D$ case, we apply a more constructive approach. The idea is very simple: We first investigate the influence of $\vec{r} \cdot$ on the gradient of our auxiliary functions $f,g,h$, and then recombine those, such that the highest polynomial order vanishes under multiplication with $\vec{r}$. 
In the following, let us assume that the functions $q_i^x$, $q_j^{y,i}$ and $q_k^{z,i,j}$ satisfy \eqref{eq:DiffFormel}. Moreover, let $\vec{r}=(x\;y\;z)^\top$.
Therefore, we first prove the following auxiliary lemmas:
\begin{lemma}\label{cor:parta}
For $i+j+k\leq p$, we have the relation
\begin{equation*}
    \vec{r} \cdot (\Grad f_i(x,y,z)) g_{i,j}(x,y,z) h_{i,j,k}(x,y,z) = i \, u_{ijk}^{\blacktriangle}(x,y,z) + \underbrace{p(x,y,z)}_{\in \Po_{p-1}}.
\end{equation*}
\end{lemma}
\begin{proof}
It suffices to show that 
$\vec{r} \cdot \Grad f_i = i \, q_{i}^x{\left(\frac{4x}{1-2y-z}\right)} {\left(\frac{1-2y-z}{4}\right)}^i + \mathcal{R}_{i-1}$, where $\mathcal{R}_{i-1} \in \Po_{i-1}$.
Using  the abbreviation $\eta = \frac{4x}{1-2y-z}$, we introduce
$\tilde{f}_i(x,y,z) = \begin{pmatrix} (q_{i}^x)'(\eta) \\ \frac12 \tilde{Q}_{i-1}^x(\eta)\\ \frac14 \tilde{Q}_{i-1}^x(\eta) \end{pmatrix}$.
Then an easy computation shows that $\Grad f_i(x,y,z) = \tilde{f}_i(x,y,z) {\left(\frac{1-2y-z}{4}\right)}^{i-1}$ by applying~\eqref{eq:DiffFormel}.
Then, 
\begin{equation*}
    \begin{aligned}
    \vec{r} \cdot \tilde{f}_i(x,y,z) &= {\left(x (q_{i}^x)'{\left(\eta\right)} + \frac{y}{2} \tilde{Q}_{i-1}^x{\left(\eta\right)} + \frac{z}{4} \tilde{Q}_{i-1}^x {\left(\eta\right)}\right)}\\
    &=\frac{1-2y-z}{4}{\left(\frac{4x}{1-2y-z} (q_{i}^x)'{\left(\eta\right)} - \tilde{Q}_{i-1}^x{\left(\eta\right)} + \frac{1}{1-2y-z}\tilde{Q}_{i-1}^x {\left(\eta\right)}\right)}.
    \end{aligned}
\end{equation*}
The last term simplifies to 
    $\vec{r} \cdot \tilde{f}_i(x,y,z) = i {\left(\frac{1-2y-z}{4}\right)} q_{i}^x{\left(\eta\right)} + \frac{1}{4}\tilde{Q}_{i-1}^x{\left(\eta\right)}$
by using~\eqref{eq:DiffFormel} again.
This gives
\begin{equation*}
    \vec{r} \cdot \Grad f_i(x,y,z) = i \underbrace{{\left(\frac{1-2y-z}{4}\right)}^i q_{i}^x {\left(\frac{4x}{1-2y-z}\right)}}_{=f_i(x,y,z)} + \frac{1}{4} \underbrace{{\left(\frac{1-2y-z}{4}\right)}^{i-1} \tilde{Q}_{i-1}^x{\left(\frac{4x}{1-2y-z}\right)}}_{\in \Po_{i-1}}
\end{equation*}
and finishes the proof.
\end{proof}
\begin{remark}
Using $\eta = \frac{4x}{1-2y-z}$,
    the proof shows that
    \begin{equation}\label{eq:Gradf}
    \Grad f_i(x,y,z) ={\left(\frac{1-2y-z}{4}\right)}^{i-1} 
    \begin{pmatrix} (q_{i}^x)'(\eta) \\ \frac12 \tilde{Q}_{i-1}^x(\eta)\\ \frac14 \tilde{Q}_{i-1}^x(\eta)\end{pmatrix}.
    \end{equation}
\end{remark}

\begin{lemma}\label{cor:partb} 
If $i+j+k\leq p$, then
\begin{equation*}
    \vec{r} \cdot f_i(x,y,z) (\Grad g_{i,j}(x,y,z)) h_{i,j,k}(x,y,z) = j \, u_{ijk}^{\blacktriangle}(x,y,z) + \underbrace{p(x,y,z)}_{\in \Po_{p-1}}.
\end{equation*}
\end{lemma}
\begin{proof}
With $\chi=\frac{2y}{1-z}$, we set
\begin{equation}\label{eq:gradg}
    \tilde{g}_{i,j}(x,y,z) = \begin{pmatrix}0\\ (q_j^{y,i})'(\chi)\\ \frac12   \tilde{Q}_{j-1}^{y,i}(\chi) \end{pmatrix}.
\end{equation}
Then, $\Grad g_{i,j}(x,y,z) = \tilde{g}_{i,j}(x,y,z) {\left(\frac{1-z}{2}\right)}^{j-1}$ by using~\eqref{eq:DiffFormel}
for the family of polynomials $q_j^{y,i}$ of degree $j$ which may depend on a parameter $i$.
We write the scalar product as follows
\begin{equation*}
\begin{aligned}
    \vec{r} \cdot \tilde{g}_{i,j}(x,y,z) &= y(q_j^{y,i})'(\chi)+ \frac{z}{2}   \tilde{Q}_{j-1}^{y,i}(\chi)\\
    &= \frac{1-z}{2} {\left(  \frac{2y}{1-z} (q_j^{y,i})'(\chi)-  \tilde{Q}_{j-1}^{y,i}(\chi)+\frac{1}{1-z} \tilde{Q}_{j-1}^{y,i}(\chi) \right)}
    &\overset{\eqref{eq:DiffFormel}}{=} j \frac{1-z}{2} (q_j^{y,i})(\chi)+\frac12 \tilde{Q}_{j-1}^{y,i}(\chi).
    \end{aligned}
\end{equation*}
A multiplication with ${\left(\frac{1-z}{2}\right)}^{j-1}$ gives
$\vec{r} \cdot \Grad g_{i,j}(x,y,z) = j \underbrace{{\left(\frac{1-z}{2}\right)}^{j} q_j^{y,i}{\left(\frac{2y}{1-z}\right)}}_{=g_{i, j}(x,y,z)}  + \underbrace{\mathcal{R}_{j-1}(x,y,z)}_{\in \Po_{j-1}}$.
\end{proof}


\begin{lemma}\label{cor:partc}
There holds
$    \vec{r} \cdot f_i(x,y,z) g_{i,j}(x,y,z) (\Grad h_{i,j,k}(x,y,z)) = k \, u_{ijk}^{\blacktriangle}(x,y,z) + \underbrace{p(x,y,z)}_{\in \Po_{p-1}}$,
for all $i+j+k\leq p$.
\end{lemma}
\begin{proof}
Note that
$\Grad h_{i,j,k}(x,y,z) = \begin{pmatrix} 0 \\ 0 \\ (q_k^{z,i,j})'(z) \end{pmatrix}$. This implies
$\vec{r} \cdot \Grad h_{i,j,k}(x,y,z) 
    \stackrel{\mathmakebox[\widthof{=}]{\eqref{eq:DiffFormel}}}{=}\quad k \underbrace{q^{z,i,j}_{k}}_{=h_{i,j,k}(x,y,z)} + \underbrace{\mathcal{R}_{k-1}}_{\in \Po_{k-1}}$.
\end{proof}
We apply the results of \cref{cor:parta,cor:partb,cor:partc} to a general polynomial function of the type
\[
    \phi_{ijk}(x,y,z) = c_1 \Grad f g h + c_2 f \Grad g h + c_3 f g \Grad h, \quad  \in {\left[\Po_{i+j+k-1}\right]}^3.
\]
Thus, the relation
$\vec{r} \cdot \phi_{ijk}(x,y,z) = (c_1\, i + c_2\, j + c_3\, k) u_{ijk}^\blacktriangle + \mathcal{R}_{i+j+k-1}$
holds. If $(c_1\, i + c_2\, j + c_3\, k) = 0$, the polynomial $\phi_{ijk}(x,y,z)$ is a \nedelec~function of first kind. One possible (non-unique) set of linear independent functions is given by
\begin{equation}\label{eq:NewTet}
    \begin{aligned}
    v_{ijk}^{\blacktriangle,II,\ned} &\coloneqq j\, \Grad (f_i) g_{i,j} h_{i,j,k} - i\, f_i \Grad(g_{i,j}) h_{i,j,k},\\
    v_{ijk}^{\blacktriangle,III, \ned} &\coloneqq k\, \Grad (f_i) g_{i,j} h_{i,j,k}- i\, f_i g_{i,j} \Grad(h_{i,j,k}),
    \end{aligned}
\end{equation}
with the auxiliary functions~\eqref{eq:defaux3d}.
We summarize our results in the following theorem.
\begin{theorem}
Let us assume that $q_i^x$, $q_j^{y,i}$ and $q_k^{z,i,j}$ satisfy \eqref{eq:DiffFormel}.
Let $i+j+k \leq p$ with $i\geq 2, j,k \geq 1$.
Then, the functions $v^{\blacktriangle,I}_{ijk}(x,y,z),v^{\blacktriangle,IV}_{1jk}(x,y,z)$ as in~\eqref{Dual:TetBasis} and $v^{\blacktriangle,II,\ned}_{ijk}(x,y,z), v^{\blacktriangle,III,\ned}_{ijk}(x,y,z)$ as in~\eqref{eq:NewTet} are a basis of the interior functions of the \nedelec space of second kind. 
Furthermore, using $v^{\blacktriangle,I}_{ijk}(x,y,z)$ only up to polynomial degree $p-1$ yields the \nedelec~functions of first kind.
\end{theorem}
\begin{proof}
It remains to show that the functions form a basis of $H(\Curl,\blacktriangle)$. This follows from the results of \cite{zaglmayr2006}, since the basis functions are just a basis transformation with a block diagonal 
matrix consisting of regular $3\times 3$ blocks.
\end{proof}

Although this is satisfied for any polynomial basis function in the auxiliary functions \eqref{eq:defaux3d}, 
which follows the construction principle in \eqref{Dual:TetBasis}, there are basis functions 
which have advantages with regard to condition number and sparsity of the element matrices.
Here, we mention the choices of  $q^x_i(t)=\hat{L}_i(t)$, $q_{j}^{y,i}(t)=\phat_j^{2i}(t)$ and $q_{k}^{z,i,j}(t)=\phat_k^{2i+2j}(t)$ in \eqref{eq:defaux3d}, giving
\begin{equation}\label{def:aux2Jac}
f_i(x,y,z)=L_i\left(\frac{4x}{1-2y-z}\right)\left(\frac{1-2y-z}{4}\right)^i,\; 
g_{i,j}(y,z)=\phat_j^{2i}\left(\frac{2y}{1-z}\right)\left( \frac{1-z}{2}\right)^j,
\textrm{and }
h_{i,j,k}(z)=\phat_k^{2i+2j}(z).
\end{equation}  
Then, the interior functions are given by
\begin{equation}\label{Dual:TetBasisJac}
    \begin{aligned}
    v_{ijk}^{\blacktriangle,I} &= \Grad (f_i) g_{i,j} h_{i,j,k} + f_i \Grad(g_{i,j}) h_{i,j,k} + f_i g_{i,j} \Grad(h_{i,j,k}),
    &\quad & v_{ijk}^{\blacktriangle,II,\ned} &= j\Grad (f_i) g_{i,j} h_{i,j,k} - i f_i \Grad(g_{i,j}) h_{i,j,k} \\
    v_{ijk}^{\blacktriangle,III,\ned} &= k \Grad (f_i) g_{i,j} h_{i,j,k}  - i f_i g_{i,j} \Grad(h_{i,j,k}),&\quad &
    v_{1jk}^{\blacktriangle,IV} &= v^{\ned_0}_{[1,2]} g_{1,j} h_{1,j,k},
    \end{aligned}
\end{equation}
for $2 \leq i, 1\leq j,k$, and $i+j+k\leq p$.
\begin{remark}
    The biorthogonal functions to the basis \eqref{Dual:TetBasisJac} can be developed similarly to the functions
    in \cite{haubold2024}. Only the matrix $A$ of Lemma 4.16 in \cite{haubold2024} has to be modified
    from $A=\begin{pmatrix} 1 & 1 & 1\\ 1 & -1 &1 \\ 1 & 1 & -1 \end{pmatrix}$ to the regular matrix
    $A=\begin{pmatrix} 1 & 1 & 1\\ j & -i &0 \\ k & 0 & -i \end{pmatrix}$.
\end{remark}

\begin{remark}
The main application of these results is basis functions which are constructed with the aid of nested 
polynomials as in \eqref{def:aux2Jac}. However, the result can also be applied to functions 
of the form \eqref{Dual:TetBasisJac}, where the one dimensional functions $f_i$, $g_{i,j}$ and
$h_{i,j,k}$ are any polynomials if  $i$, $j$ and $k$ are understood as the polynomial degrees.
In particular if $h_{ijk}$ is chosen as Lagrange polynomials of degree $p$, then $k$ has to be replaced by $p$
in the factor before $ (\nabla f_i) g_{ij} h_{ijk} $ in the definition of $v_{ijk}^{\blacktriangle,III,\ned}$ for all $k$ and so on.
\end{remark}

\section{New \texorpdfstring{$H(\Div)$}{} basis functions on a tetrahedron}\label{ch:Div+Basis}
For the case of $H(\Div)$ conforming basis functions, we must show that our construction lies in the Raviart-Thomas space. The triangular case is trivial, since the $H(\Div)$ basis functions are identical to the $H(\Curl)$ case up to a rotation by $90^\circ$.
In three space dimensions, the relation
\begin{equation}\label{eq:RT_Cond}
    \tilde{w}_{ijk}^{\blacktriangle,I}(\vec{r}) = \vec{r} \tilde{\Po}_{k-1}(\vec{x}) + {(\Po_{k-1})}^3.
\end{equation}
has to be shown.
First, note that the functions $w_{ijk}^{\blacktriangle,s}(x,y,z)$ $s\in \{I,II,III\}$ can be expressed as
$w_{ijk}^{\blacktriangle,s}(x,y,z)= a_0\, f \Grad g \times \Grad h - a_1\, g \Grad f \times \Grad h + a_2\, h \Grad f \times \Grad g$
with coefficients $a_0,a_1,a_2$. In the definition~\eqref{Div:TetBasis2}, the parameter $a_0=a_1=a_2=2$ is chosen for $s=I$,
see also~\eqref{Div:TetBasis}, \eqref{Div:TetBasis2}.
We will modify these coefficients such that the functions fit into~\eqref{eq:RT_Cond}.
Note that each summand contains the gradients of two of the auxiliary functions $f,g,h$, where the third one
does not involve derivatives. In order to obtain a structure as in~\eqref{eq:RT_Cond}, 
the polynomial (of degree $i$, $j$, or $k$) without derivative has to be transformed into the derivatives using~\eqref{eq:DiffFormel}.
Then, a corresponding factor ($i$, $j$ or $k$) appears. This motivates the choices $a_0=i$, $a_1=j$ and $a_2=k$.

Similar to the case of \nedelec~basis functions, we show that
\begin{equation}\label{eq:Defansatz}
   w_{ijk}^{\blacktriangle,I}(\vec{r}) = i \, f_i \Grad g_{i,j} \times \Grad h_{i,j,k} -
   j\, g_{i,j} \Grad f_i \times \Grad h_{i,j,k} + k \, h_{i,j,k} \Grad f_{i} \times \Grad g_{i,j}
\end{equation}
will be our Raviart-Thomas functions for $i\geq 2$.
In the case $i=1$, we introduce
\begin{equation}\label{eq:Defansatz1}
     w_{1jk}^{\blacktriangle,I}(\vec{r}) =(j+2)g_{1,j} \left(\nabla h_{1,j,k}\times v^{\ned_0}_{[1,2]}\right)
    -k h_{1,j,k} \left(\nabla  g_{1,j}\times v^{\ned_0}_{[1,2]}\right)-k g_{1,j} h_{1,j,k} \Curl v^{\ned_0}_{[1,2]}
    \end{equation}

We introduce the
abbreviations $\chi=\frac{2y}{1-z}$ and $\eta=\frac{4x}{1-2y-z}$ again.
\begin{theorem}\label{main:Hdiv}
Let us assume that $q_i^x$, $q_j^{y,i}$ and $q_k^{z,i,j}$ satisfy \eqref{eq:DiffFormel}.
    Let $w_{ijk}^{\blacktriangle,I}(\vec{r})$ be defined by~\eqref{eq:Defansatz} and \eqref{eq:Defansatz1}, respectively. Then,
    $w_{ijk}^{\blacktriangle,I}(\vec{r})\in \RT_{i+j+k-2}$. More precisely,
    $w_{ijk}^{\blacktriangle,I}(\vec{r})= p_1(x,y,z) \vec{r} +\vec{p}_2(x,y,z)$,
    where
    \[
    p_1(x,y,z)=\left\{ \begin{array}{cc}
    (q_{i}^x)'(\eta) {\left( \frac{1-2y-z}{4}\right)}^{i-1} (q_j^{y,i})'(\chi) {\left( 
    \frac{1-z}{2}\right)}^{j-1} (q_k^{z,i,j})'(z) \in \Po_{i+j+k-3} & i\geq 2 \\
    \left( \frac{1-z}{2}\right)^{j-1} \frac18 (q_k^{z,1,j})'(z) \left( (2y+z)(q_j^{y,1})'(\chi) +2(1-z) q_j^{y,1}(\chi)  \right)\in \Po_{j+k-1}& i=1.
    \end{array}\right.
    \]
    and $\vec{p}_2(x,y,z)\in {[\Po_{i+j+k-3}]}^3$ or $\vec{p}_2(x,y,z)\in {[\Po_{j+k-1}]}^3$ for $i\geq2$ and 
    $i=1$, respectively.
    \end{theorem}
The proof is split into several short lemmas.

Firstly, the gradients of the auxiliary functions are computed. 
\begin{lemma}\label{lemma:gradfgh}
Let $f_i$, $g_{i,j}$ and $h_{i,j,k}$ be defined by~\eqref{eq:defaux3d}. 
Then, we have $\Grad f_i(x,y,z) =    {\left(\frac{1-2y-z}{4}\right)}^{i-1} 
    \begin{pmatrix} (q_{i}^x)'(\eta) \\ \frac12 \tilde{Q}_{i-1}^x(\eta)\\ \frac14 \tilde{Q}_{i-1}^x(\eta) \end{pmatrix}$, $
     \Grad g_{i,j}=   {\left(\frac{1-z}{2}\right)}^{j-1}\begin{pmatrix}0\\ (q_j^{y,i})' {\left(\chi\right)}\\ 
    \frac{1}{2} \tilde{Q}_{j-1}^{y,i}(\chi)
    \end{pmatrix}$ and $\Grad h_{i,j,k}= (q_k^{z,i,j})'(z) \begin{pmatrix} 0 \\ 0 \\ 1  \end{pmatrix}$. 
\end{lemma}
\begin{proof}
    The result follows from~\eqref{eq:Gradf} and~\eqref{eq:gradg}. The last equation is trivial.
\end{proof}
The next step computes the summands in~\eqref{eq:Defansatz}.
We start with the first one.
\begin{lemma}
    With the assumptions of \Cref{lemma:gradfgh},
  \begin{eqnarray*}
    i \, f_i \Grad g_{i,j} \times \Grad h_{i,j,k} &=& {\left(x (q_{i}^x)'(\eta)+ \frac{(2y+z-1)}{4}\tilde{Q}_{i-1}^x(\eta)\right)} \\
    &&{\left( \frac{1-2y-z}{4}  \right)}^{i-1} {\left(\frac{1-z}{2}\right)}^{j-1} (q_j^{y,i})'(\chi) (q_k^{z,i,j})'(z) \begin{pmatrix} 1 \\ 0 \\ 0  \end{pmatrix}.
  \end{eqnarray*}  
\end{lemma}
\begin{proof}
We simplify 
\begin{equation*}
\begin{aligned}
    i f_i(x,y,z)=iq_i^x (\eta)  {\left( \frac{1-2y-z}{4}  \right)}^{i}& \overset{\eqref{eq:DiffFormel}}{=}
(\eta (q_{i}^x)'(\eta)- \tilde{Q}_{i-1}^x (\eta)){\left( \frac{1-2y-z}{4}  \right)}^{i}\\
&\overset{\eta=\frac{4x}{1-2y-z}}{=} {\left(x (q_i^x)'(\eta)+ \frac{2y+z-1}{4} \tilde{Q}^x_{i-1}(\eta) \right)} {\left( \frac{1-2y-z}{4}  \right)}^{i-1}.
\end{aligned}
\end{equation*}
The assertion follows now from \Cref{lemma:gradfgh} by using the cross product.
\end{proof}
\begin{lemma}\label{lemma:gnfnh}
    With the assumptions of \Cref{lemma:gradfgh},
  \begin{equation*}
    j \, g_{i,j} \Grad h_{i,j,k} \times \Grad f_{i} = {\left( y (q_j^{y,i})'(\chi)+\frac{z-1}{2} \tilde{Q}_{j-1}^{y,i}(\chi)\right)}
   {\left( \frac{1-2y-z}{4}  \right)}^{i-1} {\left(\frac{1-z}{2}\right)}^{j-1} (q_k^{z,i,j})'(z)    
   \begin{pmatrix} -\frac{\tilde{Q}_{i-1}^x(\eta)}{2} \\ (q_{i}^x)'(\eta) \\ 0  \end{pmatrix}.
  \end{equation*}  
\end{lemma}
\begin{proof}
The proof is similar to the proof of the previous one.
We simplify 
\begin{eqnarray*}
    jg_{i,j}(x,y,z)=j q_j^{y,i}(\chi)  {\left( \frac{1-z}{2}  \right)}^{j}& \overset{\eqref{eq:DiffFormel}}{=}
{\left(\chi (q_j^{y,i})' (\chi)- \tilde{Q}_{j-1}^{y,i}(\chi) \right)} {\left( \frac{1-z}{2}  \right)}^{j}\\
&\overset{\chi=\frac{2y}{1-z}}{=} {\left( y (q_j^{y,i})'(\chi)+\frac{z-1}{2} \tilde{Q}_{j-1}^{y,i}(\chi)\right)}
   {\left( \frac{1-z}{2}  \right)}^{j-1}.
\end{eqnarray*}
The assertion follows now from \Cref{lemma:gradfgh} by using the cross product.
\end{proof}
\begin{lemma}
    With the assumptions of \Cref{lemma:gradfgh},
  \begin{eqnarray*}
   k \, h_{i,j,k} \Grad f_{i} \times \Grad g_{i,j} &=& {\left( z (q_k^{z,i,j})'(z)-\tilde{Q}_{k-1}^{z,i,j}(z)\right)}\\
    &&{\left( \frac{1-2y-z}{4}  \right)}^{i-1} {\left(\frac{1-z}{2}\right)}^{j-1} 
   \begin{pmatrix} \frac14 \tilde{Q}_{i-1}^x(\eta)(\tilde{Q}_{j-1}^{y,i}(\chi)-(q_{j}^y)'(\chi) )  \\ 
   -\frac12 (q_{i}^x)'(\eta) \tilde{Q}_{j-1}^{y,i}(\chi) \\ (q_{i}^x)'(\eta) (q_{j}^{y,i})'(\chi)  \end{pmatrix}.
   \end{eqnarray*}  
\end{lemma}
\begin{proof}
    We start with $h_{i,j,k}$. Using~\eqref{eq:DiffFormel} again
        $k h_{i,j,k}(z)= k q_k^{z,i,j}(z)=
        z (q_k^{z,i,j})'(z)- \tilde{Q}_{k-1}^{z,i,j}(z)$
    with $\tilde{Q}^{z,i,j}_{k-1} \in \Po_{k-1}$.
    Using \Cref{lemma:gradfgh}, 
    \begin{equation*}
    \Grad g_{i,j}=   {\left(\frac{1-z}{2}\right)}^{j-1}\begin{pmatrix}0\\ (q_{j}^{y,i})' {\left(\chi\right)}\\ 
    \frac{1}{2} \tilde{Q}_{j-1}^{y,i} (\chi)  \end{pmatrix} \\
    \end{equation*}
    The assertion follows now from taking the cross product.
\end{proof}
Concerning the involved terms for $w^{\blacktriangle,I}_{1jk}$ similar results can be proved.
We summarize them in one lemma.
\begin{lemma}\label{lemma:grad1gh}
    Let the assumptions of \Cref{lemma:gradfgh} be satisfied. Moreover, let
$v^{\ned_0}_{[1,2]}=-\frac18 \begin{pmatrix} 1-2y-z & 2x & x\end{pmatrix}^\top$ be the \nedelec~function.
Then,  $g_{1,j} h_{1,j,k}  \Curl v^{\ned_0}_{[1,2]}, h_{1,j,k} \Grad g_{1,j} \times  v^{\ned_0}_{[1,2]},
 g_{1,j} \Grad h_{1,j,k} \times  v^{\ned_0}_{[1,2]}\in[\Po_{j+k}]^3$, where
\[
    \begin{aligned}
        k g_{1,j} h_{1,j,k}  \Curl v^{\ned_0}_{[1,2]} &= \frac{-z}{8} (q_k^{z,1,j})'(z) q_j^{y,1} \left(\frac{2y}{1-z}\right)
        \left( \frac{1-z}{2}\right)^j &\begin{pmatrix} 0 \\ -2 \\ 4 \end{pmatrix} & +\vec{v}_{1,j+k-1} ,\\
        k h_{1,j,k} \Grad g_{1,j} \times  v^{\ned_0}_{[1,2]} &= \frac{-z}{8} (q_k^{z,1,j})'(z) \left( \frac{1-z}{2}\right)^{j-1}& \begin{pmatrix} x ((q_j^{y,1})'-\tilde{Q}_{j-1}^{y,1})(\chi) ,\\
        -\frac12(2y+z) \tilde{Q}_{j-1}^{y,1}(\chi)\\ (2y+z) (q_j^{y,1})'(\chi)  \end{pmatrix}&+\vec{v}_{2,j+k-1} \\
        -(j+2) g_{1,j} \Grad h_{1,j,k} \times  v^{\ned_0}_{[1,2]} &= -\frac18 (q_k^{z,1,j})'(z)
        u(y,z) &\begin{pmatrix} 
        2x \\ 2y+z \\ 0\end{pmatrix} &+\vec{v}_{3,j+k-1}
    \end{aligned}
    \]
    with $\vec{v}_{l,s}\in[\Po_s]^3$, $l=1,2,3$ and
    $u(y,z)=\left( 2 q_j^{y,1} (\chi)\left( \frac{1-z}{2}\right)^j + (y (q_j^{y,1})'(\chi)+\frac{z}{2}\tilde{Q}_{j-1}^{y,1}(\chi)\left( \frac{1-z}{2}\right)^{j-1} \right)$.
\end{lemma}
\begin{proof}
The first two assertions are a direct consequence of \Cref{lemma:gradfgh} and the structure of the 
N\'ed\'elec function $v^{\ned_0}_{[1,2]}$ by using \eqref{eq:DiffFormel}.
For the last equation, we simplify the part $j g_{1,j}$ as in \Cref{lemma:gnfnh}.
\end{proof}

Now, we are in the position to prove the main result \Cref{main:Hdiv} of this section.
\begin{proof}
We start with the case $i\geq 2$.
    Note that 
    \[
    \mathcal{R}_2:=\tilde{Q}_{k-1}^{z,i,j}(z)
   {\left( \frac{1-2y-z}{4}  \right)}^{i-1} {\left(\frac{1-z}{2}\right)}^{j-1} 
   \begin{pmatrix} \frac14 \tilde{Q}_{i-1}^x(\eta)(\tilde{Q}_{j-1}^{y,i}(\chi)-(q_{j}^y)'(\chi) )  \\ 
   -\frac12 (q_{i}^x)'(\eta) \tilde{Q}_{j-1}^{y,i}(\chi) \\ (q_{i}^x)'(\eta) (q_{j-1}^{y,i})'(\chi)  \end{pmatrix}
    \]
    is a polynomial of degree $i+j+k-3$. Therefore, we consider
\[
w_{0,ijk}(x,y,z)=i \, f_i \Grad g_{i,j} \times \Grad h_{i,j,k} - 
   j\, g_{i,j} \Grad f_i \times \Grad h_{i,j,k} + k \, h_{i,j,k} \Grad f_{i} \times \Grad g_{i,j}+\mathcal{R}_2
\]
Using the previous lemmas
\[
w_{0,ijk}(x,y,z)={\left( \frac{1-2y-z}{4}  \right)}^{i-1} {\left(\frac{1-z}{2}\right)}^{j-1} (q_{k}^{(z,i,j)})'(z) \mathcal{R}_3(x,y,z)
\]
with
\begin{equation*}
\begin{aligned}
\mathcal{R}_3(x,y,z)&:={\left(x (q_{i}^x)'(\eta)+ \frac{(2y+z-1)}{4}\tilde{Q}_{i-1}^x(\eta)\right)}
   (q_{j}^{y,i})'(\chi)    \begin{pmatrix} 1 \\ 0 \\ 0  \end{pmatrix}\\
   &+{\left( y (q_{j}^{y,i})'(\chi)-\frac{1-z}{2} \tilde{Q}_{j-1}^{y,i}(\chi)\right)}
     \begin{pmatrix} -\frac{\tilde{Q}_{i-1}^x(\eta)}{2} \\ (q_i^{x})'(\eta) \\ 0  \end{pmatrix} 
    + 
   z\begin{pmatrix} \frac14 \tilde{Q}_{i-1}^x(\eta){(\tilde{Q}_{j-1}^{y,i}(\chi)-(q_{j}^y)'(\chi))}  \\ 
   -\frac12 (q_{i}^x)'(\eta) \tilde{Q}_{j-1}^{y,i}(\chi) \\ (q_{i}^x)'(\eta) (q_{j}^{y,i})'(\chi)  \end{pmatrix}\\
   &= (q_{i}^x)'(\eta) (q_{j}^{y,i})'(\chi) \begin{pmatrix} x \\ y \\z  \end{pmatrix} +
   \begin{pmatrix} \frac14 \tilde{Q}_{i-1}^x(\eta){(\tilde{Q}_{j-1}^{y,i}(\chi)-(q_j^{y,i})'(\chi))}  \\ 
   -\frac12 (q_i^x)'(\eta) \tilde{Q}_{j-1}^{y,i}(\chi) \\ 0    \end{pmatrix}.
\end{aligned}
\end{equation*}
This gives
\begin{equation*}
\begin{aligned}
w_{0,ijk}(x,y,z)= (q_{i}^x)'(\eta){\left( \frac{1-2y-z}{4}  \right)}^{i-1} (q_{j}^{y,i})'(\chi) {\left(\frac{1-z}{2}\right)}^{j-1} (q_{k}^{z,i,j})'(z) \vec{r}\\
+ \underbrace{{\left( \frac{1-2y-z}{4}  \right)}^{i-1} {\left(\frac{1-z}{2}\right)}^{j-1} 
\begin{pmatrix} \frac14 \tilde{Q}_{i-1}^x(\eta)(\tilde{Q}_{j-1}^{y,i}(\chi)-(q_j^{y,i})'(\chi))  \\ 
   -\frac12 (q_i^x)'(\eta) \tilde{Q}_{j-1}^{y,i}(\chi) \\ 0    \end{pmatrix}
(q_{k}^{z,i,j})'(z)}_{\in \Po_{i+j+k-3}},
\end{aligned}
\end{equation*}
and proves the main result for $i\geq 2$.
The case $i=1$ is a direct consequence of \cref{lemma:grad1gh} by ignoring the minor terms
$\vec{v}_{l,j+k-1}$, $l=1,2,3$.
\end{proof}
This construction is satisfied for any polynomial basis in the auxiliary functions \eqref{eq:defaux3d}, 
which follows the construction principle in \eqref{Div:TetBasis2}. But there are basis functions 
which have advantages in the condition number and sparsity of the element matrices.
The optimal choice with regard to sparsity are the Jacobi polynomials
\begin{equation}
\label{def:aux2Jacmod}
g_{i,j}(y,z)=\phat_j^{2i-1}\left(\frac{2y}{1-z}\right)\left( \frac{1-z}{2}\right)^j,
\textrm{and }
h_{i,j,k}(z)=\phat_k^{2i+2j-2}(z)
\end{equation}
with the Jacobi weights $2i-1$ and $2i+2j-2$ applied in the definitions of $g_{i,j}$ and $h_{i,j,k}$, respectively, \cite{beuchler2012a}.
The aim of the next section is to develop biorthogonal functions to them in closed formulas.
In order to get full biorthogonality, we have to modify the weights in the Jacobi polynomials.
In particular the choices \eqref{def:aux2Jac},  as for $H(\Curl)$, are used.
Again, the system matrix is sparse with $\mathcal{O}(p^3)$ nonzero entries, although the constant is larger than the optimal choice \eqref{def:aux2Jacmod}. Similar properties can also be observed for the condition number.\\
This leads to the following basis functions
\begin{equation}\label{eq:DefHdiv}
  \begin{aligned}
    w_{ijk}^{\blacktriangle,II,\ned}&:= \Curl v_{ijk}^{\blacktriangle,II,\ned} =i f_i \nabla g_{i,j} \times \nabla h_{i,j,k}+ 
     j g_{i,j} \nabla h_{i,j,k} \times \nabla f_i-(i+j)  h_{i,j,k} \nabla f_{i} \times \nabla g_{i,j}\\
    w_{ijk}^{\blacktriangle,III,\ned}&:=\Curl v_{ijk}^{\blacktriangle,III,\ned} =-i f_i \nabla g_{i,j} \times \nabla h_{i,j,k}+(i+k)
      g_{i,j} \nabla h_{i,j,k} \times \nabla f_i-k  h_{i,j,k} \nabla f_{i} \times \nabla g_{i,j} \\
    w_{1jk}^{\blacktriangle,IV}&:= \Curl v_{1jk}^{\blacktriangle,IV} = \Curl {\left(v^{\ned_0}_{[1,2]} g^\ned_{j} h^{\ned,j}_{k}\right)},\\
     w_{ijk}^{\blacktriangle,I}&:=i f_i \nabla g_{i,j} \times \nabla h_{i,j,k}+ 
     j g_{i,j} \nabla h_{i,j,k} \times \nabla f_i+k  h_{i,j,k} \nabla f_{i} \times \nabla g_{i,j},  
     \\ 
    w_{1jk}^{\blacktriangle,I}&:= (j+2)g^\ned_{j} \left(\nabla h^{\ned,j}_{k}\times v^{\ned_0}_{[1,2]}\right)
    -k h^{\ned,j}_{k} \left(\nabla  g^\ned_{j}\times v^{\ned_0}_{[1,2]}\right)-k g^\ned_{j} h^{\ned,j}_{k} \Curl v^{\ned_0}_{[1,2]} 
    \quad\textrm{and}\\
   w_{10k}^{\blacktriangle,I}&:= 4 w_1^{\blacktriangle,F_1}(x,y,z) \phat^4_k(z),
    \end{aligned}
\end{equation}
where $g^\ned_{j}(y,z)=\phat_j^{3}\left(\frac{2y}{1-z}\right)\left( \frac{1-z}{2}\right)^j$ and
$h^{\ned,j}_{k}(z)=\phat_k^{2j+3}(z)$.

\section{Biorthogonal functions}\label{sec:Biorthogonal}
We follow the same ansatz as in~\cite{haubold2024} to derive the biorthogonal functions to the $H(\Div)$ basis functions. 
First we derive the biorthogonal functions to the auxiliary set of basis functions and then combine them. 

We recall the auxiliary functions from \Cref{sec:newned} and their gradients of the previous section.
In the following, we investigate the functions \eqref{eq:DefHdiv}.
Using~\eqref{Div:TetBasis}, \eqref{Div:TetBasis2}, the element based basis functions of types $I$, $II$ and $III$ for $i\geq 2$ are linear combinations 
of the functions 
\[
\widetilde{w}_{ijk}^I := f_i(\Grad g_{i,j} \times \Grad h_{i,j,k}),\quad
\widetilde{w}_{ijk}^{II} := g_{i,j}(\Grad h_{i,j,k} \times \Grad f_i),  \textrm{ and }
\widetilde{w}_{ijk}^{III} := h_{i,j,k}(\Grad f_i \times \Grad g_{i,j}).
\]
In the same way, the element based functions of types $w^{\blacktriangle,s}_{ijk}$, $s\in\{I,IV\}$
are linear combinations of
\[
\widetilde{w}_{1jk}^{IV}=g_{1,j}(y,z) h_{1,j,k}(z) \Curl v^{\ned_0}_{[1,2]} + h_{1,j,k}(z) \nabla (g_{1,j}(y,z))\times v^{\ned_0}_{[1,2]} 
\quad\textrm{and}\quad
\widetilde{w}_{1jk}^{I}=g_{1,j}(y,z) \nabla (h_{1,j,k}(z)) \times v^{\ned_0}_{[1,2]} 
\]

Therefore, the functions $\widetilde{w}_{ijk}^{(s)}$, $s\in\{I,II,III\}$ and $\widetilde{w}_{1jk}^{(s)}$, $s\in\{I,IV\}$ are considered in a first step.

Then, one obtains from the computations
of the previous section the relations
\begin{equation}\label{eq:aux_def}
    \begin{aligned}
    \widetilde{w}^I_{ijk} &\coloneqq f_i(\Grad g_{i,j} \times \Grad h_{i,j,k}) =&&{\left(\frac{1-2y-z}{4}\right)}^{i} {\left(\frac{1-z}{2}\right)}^{j-1} \lhat_i{\left(\frac{4x}{1-2y-z}\right)} \begin{pmatrix}
            P_{j-1}^{(2i,0)} {\left(\frac{2y}{1-z}\right)} P_{k-1}^{(2i+2j,0)}(z)\\0\\0
        \end{pmatrix}\\
    \widetilde{w}^{II}_{ijk} &\coloneqq g_{i,j}(\Grad h_{i,j,k} \times \Grad f_i) =&&{\left(\frac{1-2y-z}{4}\right)}^{i-1} {\left(\frac{1-z}{2}\right)}^{j} \phat_{j}^{2i}{\left(\frac{2y}{1-z}\right)}P_{k-1}^{(2i+2j,0)}(z) \begin{pmatrix}
            -\frac12 L_{i-2} {\left( \frac{4x}{1-2y-z}\right)} \\ L_{i-1} {\left( \frac{4x}{1-2y-z}\right)} \\ 0
            \end{pmatrix}\\
    \widetilde{w}^{III}_{ijk} &\coloneqq h_{i,j,k}(\Grad f_i \times \Grad g_{i,j}) =&& {\left(\frac{1-2y-z}{4}\right)}^{i-1} {\left(\frac{1-z}{2}\right)}^{j-1} \phat^{2i+2j}_k(z) \\ & &&\begin{pmatrix}
            \frac12 L_{i-2}{\left(\frac{4x}{1-2y-z}\right)}{\left[ {\left(\frac{y}{1-z}-\frac12 \right)} P_{j-1}^{(2i,0)}{\left(\frac{2y}{1-z}\right)} - \frac{j}{2} \phat_{j}^{2i}{\left(\frac{2y}{1-z}\right)}\right]}\\
           -L_{i-1}{\left(\frac{4x}{1-2y-z}\right)} \left[{\left(\frac{y}{1-z}\right)} P_{j-1}^{(2i,0)}{\left(\frac{2y}{1-z}\right)} - \frac{j}{2} \phat_{j}^{2i}{\left(\frac{2y}{1-z}\right)}\right]\\
           L_{i-1}{\left(\frac{4x}{1-2y-z}\right)}P_{j-1}^{(2i,0)} {\left(\frac{2y}{1-z}\right)} .
        \end{pmatrix}.
        \end{aligned}
        \end{equation}
Note that the last functions can be decomposed as 
\begin{equation}\label{eq:aux_def3}
\widetilde{w}^{III}_{ijk} ={\left(\frac{1-2y-z}{4}\right)}^{i-1} {\left(\frac{1-z}{2}\right)}^{j-1} \phat^{2i+2j}_k(z) 
\widetilde{r}^{III}_{ijk}
\end{equation}
with
\begin{eqnarray*}
\widetilde{r}^{III}_{ijk}&=&P_{j-1}^{(2i,0)}
        {\left(\frac{2y}{1-z}\right)} 
        \begin{pmatrix}
            -\frac14 L_{i-2} {\left( \frac{4x}{1-2y-z}\right)} \\ 0 \\ L_{i-1} {\left( \frac{4x}{1-2y-z}\right)} 
        \end{pmatrix} \\
        &&+
        {\left[ {\left(\frac{y}{1-z} \right)} P_{j-1}^{(2i,0)}{\left(\frac{2y}{1-z}\right)} - \frac{j}{2} \phat_{j}^{2i}{\left(\frac{2y}{1-z}\right)}\right]}
        \begin{pmatrix}
            -\frac12 L_{i-2} {\left( \frac{4x}{1-2y-z}\right)} \\ L_{i-1} {\left( \frac{4x}{1-2y-z}\right)} \\ 0
            \end{pmatrix}.
\end{eqnarray*}
Using  $v^{\ned_0}_{[1,2]} = -\frac18 \begin{pmatrix} 1 -2y-z\\2x\\ x \end{pmatrix} $
and the weights $\alpha_1=3$, $\beta_{1j}=2j+3$ for the involved Jacobi polynomials, one obtains
\begin{equation}
\label{def:tildew1jk}
\begin{aligned}
\widetilde{w}^{I}_{1jk}&= \frac{1}{8} \phat_j^{3}\left( \frac{2y}{1-z}\right) \left( \frac{1-z}{2}\right)^j 
P_{k-1}^{(2j+3,0)}(z) \begin{pmatrix} 2x \\ 2y+z-1 \\ 0 \end{pmatrix} \\
\widetilde{w}^{IV}_{1jk}&= \phat_k^{2j+3}(z) \left( \frac{1-z}{2}\right)^{j-1} \left\{
-\frac{1-z}{8} \phat_j^3\left( \frac{2y}{1-z}\right) \begin{pmatrix} 0 \\ -1 \\ 2\end{pmatrix} 
+\frac18 \begin{pmatrix}  1-2y-z \\ 2x \\x \end{pmatrix} \times \begin{pmatrix} 
0 \\ P_{j-1}^{(3,0)}\left( \frac{2y}{1-z}\right)\\ g_{9,jk}(y,z)\end{pmatrix}
\right\}.\\
\end{aligned}
\end{equation}
with the function $g_{9,jk}=\frac12 \left( \frac{2y}{1-z} P_{j-1}^{(3,0)}\left(\frac{2y}{1-z}\right)
-j \phat^3_j \left(\frac{2y}{1-z}\right)\right)$.
Next,  the substitutions $\eta=\frac{4x}{1-2y-z}$ and $\chi=\frac{2y}{1-z}$ are applied to
the last equation. Using \eqref{eq:BeuPill17} for the last, \eqref{eq:BeuPill17a} for the second and \eqref{eq:Rec_BP} together with \eqref{eq:NedJac} for the first component allows us to simplify
\begin{equation}\label{eq:twIV}
  \widetilde{w}^{IV}_{1jk}  =\frac18 \phat_k^{2j+3}(z) \left( \frac{1-z}{2}\right)^{j} 
\begin{pmatrix} \frac{-4x}{1-z} P_{j-1}^{(2,0)}(\chi)
\\ \frac{2(2j+2)}{2j+1} P_{j-1}^{(1,0)}(\chi) + \frac{2}{2j+1} P_j^{(1,0)}(\chi) \\ -4 P_{j}^{(1,0)}(\chi) \end{pmatrix} .
\end{equation}

As we can see, the system $\lbrace \widetilde{w}^I, \widetilde{w}^{II}, \widetilde{w}^{III}\rbrace$ is an upper triangular system, which means that the biorthogonal system $\lbrace \widetilde{b}^I,\widetilde{b}^{II}, \widetilde{b}^{III} \rbrace$ is given by a lower triangular system. 
The orthogonality is meant in the sense (up to some scaling)
\[
    \langle \widetilde{w}^{l}_{i_1,j_1,k_1} , \widetilde{b}^{m}_{i_2,j_2,k_2} \rangle = \int_{\blacktriangle} w^{l}_{i_1,j_1,k_1}(x,y,z)~ \widetilde{b}_{i_2,j_2,k_2}^{m}(x,y,z)~\mathrm{d}\vec{r} = \delta_{l,m} \delta_{i_1,i_2} \delta_{j_1,j_2} \delta_{k_1,k_2}.
\]
We can directly choose $\widetilde{b}^{III}$ since only $\widetilde{w}^{III}$ has a third component, i.e.
\begin{equation*}
\widetilde{b}^{III}_{ijk} = {\left(\frac{1-2y-z}{4}\right)}^{i} {\left(\frac{1-z}{2}\right)}^{j-1} L_{i-1}{\left(\frac{4x}{1-2y-z}\right)} P_{j-1}^{(2i,0)} {\left(\frac{2y}{1-z}\right)} P^{(2i+2j-1,1)}_{k-1}(z) ~ \begin{pmatrix} 0 \\ 0 \\ 1 \end{pmatrix}.
\end{equation*}
Then, using the abbreviations $\omega_x={\left(\frac{1-2y-z}{4} \right)}$, $\omega_y={\left( \frac{1-z}{2}\right)}$,
$\eta=\frac{4x}{1-2y-z}$ and $\chi=\frac{2y}{1-z}$
\begin{equation*}
\begin{aligned}
    \langle \widetilde{w}^{III}_{i_1,j_1,k_1} , \widetilde{b}^{III}_{i_2,j_2,k_2} \rangle=&\int_{\blacktriangle} \widetilde{w}^{III}_{i_1,j_1,k_1}(x,y,z)~ \widetilde{b}_{i_2,j_2,k_2}^{III}(x,y,z)~\mathrm{d}\vec{r} \\
    =&\int_{\blacktriangle} \omega_x^{i_1+i_2-1} L_{i_1-1}(\eta) L_{i_2-1}(\eta) \omega_y^{j_1+j_2-2} 
    P_{j_1-1}^{(2i_1,0)}(\chi) P_{j_2-1}^{(2 i_2,0)}(\chi) \hat{P}_{k_1}^{2i_1+2j_1}(z) P_{k_2-1}^{(2i_2+2j_2-1,1)}(z) \\
    \overset{\textrm{Duffy}}{=} &\int_{-1}^1 L_{i_1-1}(\eta) L_{i_2-1}(\eta) \;\mathrm{d}\eta 
    \int_{-1}^1 {\left(\frac{1-\chi}{2}\right)}^{i_1+i_2-1+1} P_{j_1-1}^{(2i_1,0)}(\chi) P_{j_2-1}^{(2 i_2,0)}(\chi) \;\mathrm{d}\chi\\
    &\int_{-1}^1 \omega_y^{i_1-1+i_2+j_1-1+j_2-1+2} \hat{P}_{k_1}^{2i_1+2j_1}(z) P_{k_2-1}^{(2i_2+2j_2-1,1)}(z) \;\mathrm{d}z
    \end{aligned}
\end{equation*}
Using the orthogonality of the Legendre and Jacobi polynomials~\eqref{Orthogonality},
the first two integrals are simplified to
\begin{equation*}
    \begin{aligned}
        \langle \widetilde{w}^{III}_{i_1,j_1,k_1} , \widetilde{b}^{III}_{i_2,j_2,k_2} \rangle&=&
        \delta_{i_1,i_2}  \int_{-1}^1 {\left(\frac{1-\chi}{2}\right)}^{2i_1} P_{j_1-1}^{(2i_1,0)}(\chi) 
        P_{j_2-1}^{(2 i_1,0)}(\chi) \;\mathrm{d}\chi \\
        &&\int_{-1}^1 \omega_y^{i_1-1+i_2+j_1-1+j_2-1+2} \hat{P}_{k_1}^{2i_1+2j_1}(z) P_{k_2-1}^{(2i_2+2j_2-1,1)}(z) \;\mathrm{d}z \\
        &=&\delta_{i_1,i_2} \delta_{j_1,j_2}  \int_{-1}^1 \omega_y^{2i_1+2j_1-1} \hat{P}_{k_1}^{2i_1+2j_1}(z) P_{k_2-1}^{(2i_1+2j_1-1,1)}(z) \;\mathrm{d}z.
    \end{aligned}
\end{equation*}
Finally, the relation~\eqref{dual:IntLeg} is applied.
Together with~\eqref{Orthogonality2} one easily obtains the relation
\begin{equation}
    \label{Ortho:33}
    \begin{aligned}
    \langle \widetilde{w}^{l}_{i_1,j_1,k_1} , \widetilde{b}^{III}_{i_2,j_2,k_2} \rangle&=
    \delta_{i_1,i_2} \delta_{j_1,j_2}  \int_{-1}^1 {\left(\frac{1-z}{2}\right)}^{2i_1+2j_1-1} \frac{1+z}{2} P_{k_1-1}^{(2i_1+2j_1-1,1)}(z) P_{k_2-1}^{(2i_1+2j_1-1,1)}(z) \;\mathrm{d}z \\
    &=
    c_{i_1,j_1,k_1,III} \delta_{i_1,i_2} \delta_{j_1,j_2} \delta_{k_1,k_2} \delta_{m,III}  \quad \textrm{for } m=I,II,III.
    \end{aligned}
\end{equation}
Next, we derive $\widetilde{b}^{II}$. Due to the orthogonality condition, we know that the second component ${(\widetilde{b}^{II}_{ijk})}_2$ is given by
\begin{equation*}
    {(\widetilde{b}^{II}_{ijk})}_2 = {\left(\frac{1-2y-z}{4}\right)}^{i-1} {\left(\frac{1-z}{2}\right)}^{j} L_{i-1}{\left(\frac{4x}{1-2y-z}\right)} P_{j-1}^{(2i-1,1)}{\left(\frac{2y}{1-z}\right)} P_{k-1}^{(2i+2j,0)}(z).
\end{equation*}
Using the same arguments as in the proof of~\eqref{Ortho:33}, it follows that
\[
        \langle {(\widetilde{w}^{II}_{i_1,j_1,k_1})}_2 , {(\widetilde{b}^{II}_{i_2,j_2,k_2})}_2 \rangle =
   c_{i_1,j_1,k_1,II} \delta_{i_1,i_2} \delta_{j_1,j_2} \delta_{k_1,k_2}.
    \]

The first component needs to be zero, due to orthogonality with $w^{I}$, and the third component needs to be chosen such that
\begin{equation}\label{eq:b2}
\begin{aligned}
    0 = \langle \widetilde{w}_{i_1,j_1,k_1}^{III},\widetilde{b}_{i_2,j_2,k_2}^{II} \rangle = 
    \int_\blacktriangle & \omega_x^{i_1-1} \omega_y^{j_1-1} L_{i_1-1}(\eta) \phat^{2i_1+2j_1}_{k_1}(z) 
    \\&\cdot \underbrace{\biggl(\frac12 \left[-\frac{\chi}{2} P_{j-1}^{(2i_1,0)}(\chi) + j_1 \phat_{j_1}^{2i_1}(\chi) \right] {(\widetilde{b}_{i_2j_2k_2}^{II})}_2 + P_{j_1-1}^{(2i_1,0)}(\chi){(\widetilde{b}_{i_2j_2k_2}^{II})}_3\biggr)}_{:=I^{II,III}_{i_1j_1k_1,i_2j_2k_2}} \mathrm{d} \vec{r}
    \end{aligned}
\end{equation}
We follow the same arguments as in~\cite{haubold2024}. Let
\begin{equation*}
    {(\widetilde{b}_{ijk}^{II})}_3 = {\left(\frac{1-2y-z}{4}\right)}^{i-1} {\left(\frac{1-z}{2}\right)}^{j} L_{i-1}{\left(\frac{4x}{1-2y-z}\right)}  P_{k-1}^{(2i+2j,0)}(z) \widehat{d}_{ijk}(y,z),
\end{equation*}
where $\widehat{d}_{ijk}(y,z)$ is chosen such that the term $I^{II,III}_{i_1j_1k_1,i_2j_2k_2}$ in \eqref{eq:b2} becomes zero. 
Next, the definition of ${(\widetilde{b}_{ijk}^{II})}_2$ is inserted into the bracket.
Using the Duffy transformation, the tensorial structure and the orthogonality~\eqref{Orthogonality2} in $\eta$ direction,
one obtains
\[
I^{II,III}_{i_1j_1k_1,i_2j_2k_2}=0 \quad\forall i_1\neq i_2.
\]
Moreover, the integration in $z$-direction  involve the same factors in ${(\widetilde{b}_{ijk}^{II})}_2$ and 
${(\widetilde{b}_{ijk}^{II})}_3$.
Therefore,
we are able to reduce the problem to the integration in $y$ direction, or after Duffy-transformation, in $\chi$ direction.
This yields the following condition
\begin{equation*}
    0=\int_{-1}^1 {\left(\frac{1-\chi}{2}\right)}^{i}\biggl(\left[-\frac{\chi}{2} P_{j_1-1}^{(2i,0)}{\left(\chi\right)} + \frac{j_1}{2} \phat_{j_1}^{2i}{\left(\chi\right)} \right] {\left(\frac{1-\chi}{2}\right)}^{i-1} P_{j_2-1}^{(2i-1,1)}{\left(\chi\right)} + P_{j_1-1}^{(2i,0)}{\left(\chi\right)}\widehat{d}_{ij_2k_2}(\chi)\biggr) \;\mathrm{d}\chi,
\end{equation*}
with $i_1=i_2=i$. This motivates the ansatz
$\widehat{d}_{i_2j_2k_2}(\chi) = \frac{\chi}{2} {\left(\frac{1-\chi}{2}\right)}^{i-1} P_{j_2-1}^{(2i-1,1)}(\chi) - c_{ij_2} {\left(\frac{1-\chi}{2}\right)}^{i} P_{j_2-1}^{(2i,0)}(\chi)$. 
Here, the first term matches the first part of the square bracket.
Concerning the second part of the square bracket, the condition
\begin{equation*}
\begin{aligned}
0=&\int_{-1}^1 {\left(\frac{1-\chi}{2}\right)}^{2i-1} \frac{j_1}{2} \phat_{j_1}^{2i}\left(\chi\right) P_{j_2-1}^{(2i-1,1)}(\chi)-
c_{ij_2}{\left(\frac{1-\chi}{2}\right)}^{2i}P_{j_1-1}^{(2i,0)}{\left(\chi\right)}  P_{j_2-1}^{(2i,0)}(\chi) \;\mathrm{d}\chi\\
\overset{\eqref{dual:IntLeg}}{=}& \int_{-1}^1 {\left(\frac{1-\chi}{2}\right)}^{2i-1} \left(\frac{1+\chi}{2}\right) P_{j_1-1}^{(2i-1,1)}\left(\chi\right) P_{j_2-1}^{(2i-1,1)}(\chi)-
c_{ij_2}{\left(\frac{1-\chi}{2}\right)}^{2i}P_{j_1-1}^{(2i,0)}{\left(\chi\right)}  P_{j_2-1}^{(2i,0)}(\chi)\;\mathrm{d}\chi
\end{aligned}
\end{equation*}
has to be satisfied. 
If $j_1\neq j_2$, the term is zero due to the orthogonality relation of the Jacobi polynomials. If $j_1=j_2$, the term is zero if
one chooses $c_{i,j_2} = \frac{j_2}{2i +j_2 -1}$.
This yields the final results
\begin{equation}
    \label{Ortho:22}
    \begin{aligned}
    \langle \widetilde{w}^{II}_{i_1,j_1,k_1} , \widetilde{b}^{II}_{i_2,j_2,k_2} \rangle&= c_{i_1,j_1,k_1,II}\delta_{i_1,i_2} \delta_{j_1,j_2} 
    \delta_{k_1,k_2} \textrm{ and} \quad
    \langle \widetilde{w}^{III}_{i_1,j_1,k_1} , \widetilde{b}^{II}_{i_2,j_2,k_2} \rangle&= 0.
    \end{aligned}
\end{equation}
The third biorthogonal function $\widetilde{b}^{I}_{ijk}$ can be derived in the same manner. The first component can be chosen due to the orthogonality relation $\langle \widetilde{w}^{I}_{ijk},\widetilde{b}^{I}_{i_2j_2k_2}\rangle$, i.e.
\begin{equation*}
    {(b^{I}_{ijk})}_1 ={\left(\frac{1-2y-z}{4}\right)}^{i-1} {\left(\frac{1-z}{2}\right)}^{j} P_{i-2}^{(1,1)}{\left(\frac{4x}{1-2y-z}\right)} P_{j-1}^{(2i,0)} {\left(\frac{2y}{1-z}\right)} P_{k-1}^{(2i+2j,0)}(z)
\end{equation*}
For the second component we apply \cref{lemma:ProdIntJacLeg}.
Our aim is to build a function which becomes orthogonal to all $\tilde{w}_{ijk}^{II}$ and 
$\tilde{w}_{ijk}^{III}$. 
Note that both functions can be written as
\begin{equation*}
\widetilde{w}_{ijk}^{II/III}=r_{jk}^{(i)}(y,z) \begin{pmatrix}
            -\frac12 L_{i-2} {\left( \frac{4x}{1-2y-z}\right)} \\ L_{i-1} {\left( \frac{4x}{1-2y-z}\right)} \\ 0
            \end{pmatrix}+s_{jk}^{(i)}(y,z) \begin{pmatrix}
            -\frac14 L_{i-2} {\left( \frac{4x}{1-2y-z}\right)} \\ 0 \\ L_{i-1} {\left( \frac{4x}{1-2y-z}\right)} 
            \end{pmatrix}
\end{equation*}
with some polynomials $r_{jk}^{(i)}(y,z)$ and $s_{jk}^{(i)}(y,z)$, where $s_{jk}^{(i)}=0$ for functions of type $II$, see \eqref{eq:aux_def3}. 
Our aim is to use the corollary \ref{lemma:ProdIntJacLeg} for the integration in $x$-direction.
\begin{lemma}\label{lemma:Vektor}
With some polynomial $t_{jk}^{(i)}(y,z)$, let
\[
\widetilde{b}^{(h)}_{ijk}=t_{jk}^{(i)}(y,z)\begin{pmatrix}
    2i~P_{i-2}^{(1,1)}{\left(\frac{4x}{1-2y-z}\right)} & (i+1) P_{i-1}^{(1,1)}{\left(\frac{4x}{1-2y-z}\right)} & \frac{1}{2}(i+1) P_{i-1}^{(1,1)}{\left(\frac{4x}{1-2y-z}\right)}  \end{pmatrix}^\top.
\]
Then, $\langle \widetilde{w}_{i_1,j_1,k_1}^{II/III}, \widetilde{b}^{(h)}_{i_2,j_2,k_2}\rangle =0$ for all $i_2\geq 2$.
\end{lemma}
\begin{proof}
A direct computation gives
\begin{equation*}
\label{eq:BiortI+II}
\begin{aligned}
 \langle \widetilde{w}^{II}_{i_1,j_1,k_1} , \widetilde{b}^{(h)}_{i_2,j_2,k_2} \rangle &\overset{\textrm{Duffy}}{=}&&
 \int_{-1}^1 -\frac{1}{2} L_{i_1-2}(\eta) 2i_2~P_{i_2-2}^{(1,1)} (\eta) +  L_{i_1-1}(\eta) (i_2+1) P_{i_2-1}^{(1,1)} (\eta)\;\mathrm{d}\eta\\
 &&&\int_{-1}^1 g_1(\chi)\;\mathrm{d}\chi \int_{-1}^1 g_2(z)\;\mathrm{d}z \\
 &\overset{\textrm{Cor. \ref{lemma:ProdIntJacLeg}}}{=} & 0 &,
\end{aligned}
\end{equation*}
where $g_1$ and $g_2$ are some polynomials.
The general case with $s_{jk}^{(i)}\neq 0$ follows by the same arguments.
\end{proof}
\begin{remark}\label{rem:BiII+III}
The result is based on the first orthogonality relation of Corollary \ref{lemma:ProdIntJacLeg}.
Using the second relation with $L_1(\eta)=\eta$ and $L_0(\eta)=1$, 
also for polynomials of the form
\begin{equation}\label{eq5p12}
\widetilde{w}^{II/III}_{1jk}=r_{jk}^{(i)}(y,z) \begin{pmatrix}
            -\frac{\eta}{2} \\ 1 \\ 0
            \end{pmatrix}+s_{jk}^{(i)}(y,z) \begin{pmatrix}
            -\frac{\eta}{4}  \\ 0 \\ 1 
            \end{pmatrix}
\end{equation}
the result
$\langle \widetilde{w}_{1,j_1,k_1}^{II/III}, \widetilde{b}^{(h)}_{i_2,j_2,k_2}\rangle =0$ with $i_2\geq 2$ follows.
\end{remark}

The application of  Lemma \ref{lemma:Vektor} motivates the choice
\[
\widetilde{b}^{I}_{ijk}(x,y,z) = {\left(\frac{1-2y-z}{4}\right)}^{i-1} {\left(\frac{1-z}{2}\right)}^{j} P_{j-1}^{(2i,0)} {\left(\frac{2y}{1-z}\right)} P_{k-1}^{(2i+2j,0)}(z) \begin{pmatrix}
    2i~P_{i-2}^{(1,1)}{\left(\frac{4x}{1-2y-z}\right)} \\ (i+1) P_{i-1}^{(1,1)}{\left(\frac{4x}{1-2y-z}\right)} \\ \frac{1}{2}(i+1) P_{i-1}^{(1,1)}{\left(\frac{4x}{1-2y-z}\right)}  \end{pmatrix}.
\]
Finally, the functions are scaled such that $\langle \widetilde{w}^{m}_{ijk},\widetilde{b}^{m}_{ijk}\rangle$
does not depend on $m$. With $\eqref{Orthogonality}$, the resulting biorthogonal functions  (up to some factor) are given by the relations
\begin{equation}\label{eq:aux_bi}
\begin{aligned}
    -2\widetilde{b}^{I}_{ijk}(x,y,z) &= {\left(\frac{1-2y-z}{4}\right)}^{\checka{i-1}} {\left(\frac{1-z}{2}\right)}^{\checka{j}} P_{j-1}^{(2i,0)} {\left(\frac{2y}{1-z}\right)} P_{k-1}^{(2i+2j,0)}(z) \begin{pmatrix}
    2i~P_{i-2}^{(1,1)}{\left(\frac{4x}{1-2y-z}\right)} \\ (i+1) P_{i-1}^{(1,1)}{\left(\frac{4x}{1-2y-z}\right)} ,\\ \frac{1}{2}(i+1) P_{i-1}^{(1,1)}{\left(\frac{4x}{1-2y-z}\right)}  \end{pmatrix}\\
    \frac{\widetilde{b}^{II}_{ijk}(x,y,z)}{2i+j-1} &= {\left(\frac{1-2y-z}{4}\right)}^{i-1} {\left(\frac{1-z}{2}\right)}^{j} L_{i-1}{\left(\frac{4x}{1-2y-z}\right)}  P_{k-1}^{(2i+2j,0)}(z) \begin{pmatrix} 0 \\ P_{j-1}^{(2i-1,1)}\left(\frac{2y}{1-z}\right) \\ \tilde{R}_j(y,z)\end{pmatrix},\\
    \frac{\widetilde{b}^{III}_{ijk}(x,y,z)}{2i+2j+k-1} &= {\left(\frac{1-2y-z}{4}\right)}^{i} {\left(\frac{1-z}{2}\right)}^{j-1} L_{i-1}{\left(\frac{4x}{1-2y-z}\right)} P_{j-1}^{(2i,0)} {\left(\frac{2y}{1-z}\right)} P^{(2i+2j-1,1)}_{k-1}(z)\begin{pmatrix} 0 \\ 0 \\ 1 \end{pmatrix}
\end{aligned}
\end{equation}
with $\tilde{R}_j(y,z)=\frac{y}{1-z} P_{j-1}^{(2i-1,1)}{\left( \frac{2y}{1-z}\right)} - \frac{j}{2i+j-1}
\frac{1-2y-z}{2(1-z)} P_{j-1}^{(2i,0)}\left(\frac{2y}{1-z}\right)$.
We summarize the results in the following theorem:
\begin{theorem}
    The auxiliary basis function $\widetilde{w}^{I}_{ijk}(x,y,z),\widetilde{w}^{II}_{ijk}(x,y,z)$ and $\widetilde{w}^{III}_{ijk}(x,y,z)$ as in~\eqref{eq:aux_def} and the functions $\widetilde{b}^{I}_{ijk}(x,y,z),\widetilde{b}^{II}_{ijk}(x,y,z)$ and $\widetilde{b}^{III}_{ijk}(x,y,z)$ as in~\eqref{eq:aux_bi} are biorthogonal to each other, i.e.
    \begin{equation}\label{eq:biI-III}
        \langle \widetilde{w}^{m_1}_{i_1j_1k_1},\widetilde{b}^{m_2}_{i_2j_2k_2}\rangle = \int_\blacktriangle \tilde{w}^{m_1}_{i_1j_1k_1}(x,y,z) \cdot b^{m_2}_{i_2j_2k_2}(x,y,z)~\mathrm{d}\vec{x} = c_{i_1,j_1,k_1,m_1} \delta_{m_1,m_2} \delta_{i_1,i_2} \delta_{j_1,j_2} \delta_{k_1,k_2},
    \end{equation}
where $c_{ijkm}=\frac{16}{(2i-1)(2i+2j-1)(2i+2j+2k-1)}$ for $m=I,II,III$.
\end{theorem}
\begin{proof}
The result follows from the previous arguments, see \eqref{Ortho:22}, \eqref{Ortho:33}.
\end{proof}

We still need to show that the auxiliary basis functions $\widetilde{w}^{I}_{1jk},\widetilde{w}^{IV}_{1jk}$  and $w^\blacktriangle_{10k}$ are biorthogonal to the basis functions $\widetilde{b}^I_{ijk},\widetilde{b}^{II}_{ijk}, \widetilde{b}^{III}_{ijk}$.
Note that the lowest order Raviart-Thomas functions are given by 
\begin{equation*}
    \begin{aligned}
               w_1^{\blacktriangle,F_1}(x,y,z)&=-\frac18 \begin{pmatrix}-x\\-y\\ 1-z \end{pmatrix} =  \frac18 \frac{1-z}{2}\left(\chi \begin{pmatrix} -\frac12\eta \\ 1 \\0 \end{pmatrix}+
        2\begin{pmatrix} \frac14 \eta \\ 0 \\ -1\end{pmatrix}
        \right)
    \end{aligned}.
\end{equation*}
The following result will be used several times.
\begin{lemma}
Let 
\begin{equation}\label{eq5p14}
\widetilde{w}^{(h)}_{ijk}(x,y,z)=\begin{pmatrix} g_{1,ijk}(x,y,z) \\ g_{2,ijk}(y,z) \\ g_{3,ijk}(y,z)\end{pmatrix}
\end{equation}
be some polynomial, where the $x$ dependence is only in the first direction.
Then, we have
\begin{equation}\label{eq5p13}
\langle \widetilde{w}^{(h)}_{i_1,j_1,k_1}, \widetilde{b}_{i_2,j_2,k_2}^{II,III} \rangle = 0
\quad\textrm{for all } i_2\geq 2.
\end{equation}
\end{lemma}
\begin{proof}
The result follows from the Duffy trick
\[
\langle \widetilde{w}^{(h)}_{i_1,j_1,k_1}, \widetilde{b}_{i_2,j_2,k_2}^{II,III} \rangle = 
\int_{-1}^1 L_{i_2-1}(\eta) \;\mathrm{d}\eta \int_{-1}^1 g_4(\chi) \;\mathrm{d}\chi
\int_{-1}^1 g_5(z) \;\mathrm{d}z=0
\]
due the orthogonality of the Legendre polynomials with some polynomials $g_5$ and $g_4$.
\end{proof}
Now, the orthogonalities can be proved:
\begin{itemize}
\item We start  with the orthogonality of $w^\blacktriangle_{10k}$ to \eqref{eq:aux_bi}.
The orthogonality to the functions of types $II$ and $III$ follows from \eqref{eq5p13} since
$w^\blacktriangle_{10k}=\hat{P}_k^3(z) w_1^{\blacktriangle,F_1}(x,y,z)$ is of the form \eqref{eq5p14}.
Due to the form of $w_1^{\blacktriangle,F_1}(x,y,z)$, the basis function $w^\blacktriangle_{1,0,k}$
is also of the form \eqref{eq5p12}. The orthogonality to the biorthogonal functions of type $I$ 
is a direct consequence of Remark \ref{rem:BiII+III}.
\item The orthogonality of $\widetilde{w}^{IV}_{1jk}$ follows by similar arguments.
Using \eqref{def:tildew1jk}, it is of the form \eqref{eq5p14}.
The orthogonality to the functions of types $II$ and $III$ follows again from \eqref{eq5p13}.
Concerning the orthogonality to the functions of type $I$, we note that 
\begin{equation}\label{eq5p15}
\begin{pmatrix} 0 \\-1 \\ 2  \end{pmatrix} \cdot \begin{pmatrix}
    2i~P_{i-2}^{(1,1)}(\eta) \\ (i+1) P_{i-1}^{(1,1)}(\eta) \\ \frac{1}{2}(i+1) P_{i-1}^{(1,1)}(\eta )\end{pmatrix} =0.
\end{equation}
Moreover, the second summand is  of the form \eqref{eq5p12}. The orthogonality to the biorthogonal functions of type $I$ 
is a direct consequence of Remark \ref{rem:BiII+III} together with \eqref{eq5p15}.
\item It remains to check the orthogonality of $\widetilde{w}^{I}_{1jk}$ to \eqref{eq:aux_bi}
which follows by the same arguments.
\end{itemize}

It remains to write down the low order functions dual to $w^{\blacktriangle}_{10k}$ and
$\widetilde{w}_{1jk}^{I/IV}$. Both are of the form
\[
\widetilde{w}^{(h)}_{ijk}(x,y,z)=\begin{pmatrix} 0 \\ g_{2,ijk}(y,z) \\ g_{3,ijk}(y,z)\end{pmatrix}
\]
which gives the orthogonality to all $\widetilde{b}^{\ell}_{ijk}$ with $i\geq 2$, $\ell\in\{I,II,III\}$.
The dual functions to the $w^{\blacktriangle}_{10k}$ are given by $\widetilde{b}_{10k}^{II}$ in \eqref{eq:aux_bi}.
The other functions are not special cases of the functions \eqref{eq:aux_bi}.
Here, some modifications in the weights are necessary.
Namely, we introduce our dual functions via the expressions
\begin{equation*}
\begin{aligned}
\widetilde{b}_{1jk}^{IV}(x,y,z)&= (2j+k+2) P_{j}^{(1,0)}\left(\frac{2y}{1-z}\right) \left(\frac{1-z}{2}\right)^j P_{k-1}^{(2j+2,1)}(z)\begin{pmatrix}0\\0\\1
\end{pmatrix} 
\quad \textrm{and} \quad \\
\widetilde{b}_{1jk}^I&=(j+2)\left(\frac{1-z}{2}\right)^{j} p_{k-1}^{(2j+3,0)}(z) 
\begin{pmatrix} 0 \\  P_{j-1}^{(2,1)}\left(\frac{2y}{1-z}\right) \\
 \frac12 P_{j-1}^{(2,1)}\left(\frac{2y}{1-z}\right)+ P_{j}^{(1,1)}\left(\frac{2y}{1-z}\right)\end{pmatrix}.
 \end{aligned}
\end{equation*}
\begin{itemize}
    \item 
    The orthogonality to the basis functions $\widetilde{w}_{ijk}^{\ell}$, $\ell=I,II,III$, $i\geq 2$
follows by the orthogonality of the Legendre polynomials.
\item By construction, the functions $\widetilde{b}_{1jk}^{IV}$ are orthogonal to $\widetilde{w}_{1j'k'}^{I}$.
\item The relation
\begin{equation}\label{eq:biIV}
\langle \widetilde{b}_{1jk}^{IV}, \widetilde{w}_{1j'k'}^{IV} \rangle=
-\frac{2}{(j+1)(j+k+1)}\delta_{j,j'}\delta_{k,k'}
\end{equation}
follows from the definition of $\widetilde{b}_{1jk}^I$, \eqref{eq:twIV} and the orthogonality relations.
\item In the same way, 
\begin{equation}\label{eq:biI}
\langle \widetilde{b}_{1jk}^{I}, \widetilde{w}_{1j'k'}^{I} \rangle= -\frac{2}{(j+1)(j+k+1)}\delta_{j,j'}\delta_{k,k'}
\end{equation}
is proved.
\item It remains to check $\langle \widetilde{b}_{1jk}^{I}, \widetilde{w}_{1j'k'}^{IV} \rangle=0$.
Here, using the definition of the functions and the Duffy trick with $\chi=\frac{2y}{1-z}$ gives
a factorization into the one dimensional integrals
\begin{eqnarray*}
\langle \widetilde{b}_{1jk}^{I}, \widetilde{w}_{1j'k'}^{IV} \rangle&=& S_{j,j'}
\int_{-1}^1 \phat_{k'}^{2j'+3}(z) p_{k-1}^{(2j+3,0)}(z)  (1-z)^{j+j'+2} \dz 
\end{eqnarray*}
with
\begin{eqnarray*}
S_{j,l}&=&\int_{-1}^1 (1-\chi) \left(\frac{2}{2l+1} \left( 2(l+1)P_{l-1}^{(1,0)}(\chi)+P_{l}^{(1,0)}(\chi) \right) 
P_{j-1}^{(2,1)}(\chi)\right.\\
&&\left.-4P_{l}^{(1,0)}(\chi) \left(\frac12 P_{j-1}^{(2,1)}+P_j^{(1,1)}(\chi) \right) \right)
\;\mathrm{d}\chi \\
&=& \frac{4}{2l+1} \int_{-1}^1 (1-\chi) \underbrace{\left((l+1)P_{l-1}^{(1,0)}(\chi)-l P_{l}^{(1,0)}(\chi)\right)}_{=\frac{2l+1}{2} (1-\chi)P_{l-1}^{(2,0)} } P_{j-1}^{(2,1)}(\chi) \;\mathrm{d}\chi \\
&&-4 \int_{-1}^1 (1-\chi) P_{l}^{(1,0)}(\chi) P_{j}^{(1,1)}(\chi) \;\mathrm{d}\chi \\
&=&  2\int_{-1}^1 (1-\chi)^2 P_{l-1}^{(2,0)}(\chi) P_{j-1}^{(2,1)}(\chi) \;\mathrm{d}\chi 
-4 \int_{-1}^1 (1-\chi) P_{l}^{(1,0)}(\chi) P_{j}^{(1,1)}(\chi) \;\mathrm{d}\chi \\
&=& 0.
\end{eqnarray*}
by using  \eqref{eq:NedJac} and Corollary \ref{folg}.
This proves the result $\langle \widetilde{b}_{1jk}^{I}, \widetilde{w}_{1j'k'}^{IV} \rangle=0$.
\end{itemize}
In a last step, we have to determine the biorthogonal functions of 
the original basis to \eqref{Dual:TetBasisJac}.
This can be done with the aid of Lemma 4.3 of \cite{haubold2024}.
Note that
\[
\begin{pmatrix} w^{\blacktriangle,I}_{ijk} \\ w^{\blacktriangle,II}_{ijk} \\ w^{\blacktriangle,III}_{ijk}
\end{pmatrix}=\begin{pmatrix} i & j & k \\ i &j & -(i+j) \\ -i & i+k & -k\end{pmatrix}
\begin{pmatrix}  \widetilde{w}_{ijk}^{I}\\  \widetilde{w}_{ijk}^{II} \\  \widetilde{w}_{ijk}^{III} \end{pmatrix}
\quad\textrm{and}\quad
\begin{pmatrix} w^{\blacktriangle,I}_{1jk} \\ w^{\blacktriangle,IV}_{1jk} \end{pmatrix}=\begin{pmatrix} j+2 & -k
\\ 1 &1 \end{pmatrix}
\begin{pmatrix}  \widetilde{w}_{1jk}^{I}\\  \widetilde{w}_{1jk}^{IV} 
\end{pmatrix}
\]
Using the scaling \eqref{eq:biI-III} and \eqref{eq:biI}, \eqref{eq:biIV}, we introduce
\begin{equation}\label{eq:finalBi}
\begin{pmatrix}
b^{\blacktriangle,I}_{ijk} \\ b^{\blacktriangle,II}_{ijk} \\
b^{\blacktriangle,III}_{ijk}\end{pmatrix}
=\tilde{c}_{i,j,k}\begin{pmatrix}
i & i & i \\
k & 0 & -i \\
-j & i & 0
\end{pmatrix}
\begin{pmatrix}
\widetilde{b}_{ijk}^{I} \\ \widetilde{b}_{ijk}^{II} \\ \widetilde{b}_{ijk}^{III}
\end{pmatrix}
\textrm{ and }
\begin{pmatrix}
b^{\blacktriangle,I}_{1jk} \\ b^{\blacktriangle,IV}_{1jk} \end{pmatrix}
=-\frac{(j+1)(j+k+1)}{2(j+k+2)}\begin{pmatrix}
    1 & -1\\
    k & j+2 
\end{pmatrix}
\begin{pmatrix}
\widetilde{b}_{ijk}^{I} \\ \widetilde{b}_{ijk}^{IV} 
\end{pmatrix}
\end{equation}
with the values $\tilde{c}_{i,j,k}=\frac{(2i-1)(2i+2j-1)(2i+2j+2k-1)}{16i(i+j+k)}$.
We summarize all observations in the following theorem.
\begin{theorem}
Let $b_{1jk}^{\blacktriangle,\ell}$, $\ell=I,IV$ and
$b^{\blacktriangle,\ell}_{ijk}(x,y,z)$, $\ell=I,II,III$ be defined by ~\eqref{eq:finalBi}.
Finally, let $-\frac{(2k+3)(k+3)}{8}b_{10k}^{\blacktriangle}=p_{k-1}^{(3,1)}(z) \begin{pmatrix} 0 \\ 0 \\1  \end{pmatrix}$.
Then, they are biorthogonal to the interior basis functions defined by \eqref{eq:DefHdiv}.
\end{theorem}

\begin{remark}
    The choice of the Jacobi weights in \eqref{eq:aux_bi} is minimal. 
    The optimal sparsity choice of the Jacobi weights $2i-1$ and $2i+2j-2$ in \eqref{eq:aux_def}
    implies an reduction of the exponent of the weight functions in \eqref{eq:aux_bi}. But this would also lead to the appearance of poles in \eqref{eq:aux_bi} and thus the biorthogonal functions would become rational. In this case, to the knowledge of the authors, there is no biorthogonal polynomial in closed form.\\
    Using the sparsity optimal choice and \eqref{eq:finalBi} leads to a very sparse system. 
\end{remark}

\bibliographystyle{plain}
\bibliography{Biorthogonal}

\end{document}